\documentclass[11pt]{amsart}
\usepackage{amssymb,amsfonts,amsthm,amsmath,enumerate,amscd}
\usepackage[english]{babel}
\usepackage{colordvi,color}

\newtheorem{theorem}{Theorem}[section]

\newtheorem{lemma}[theorem]{Lemma}
\newtheorem{corollary}[theorem]{Corollary}
\newtheorem{claim}[theorem]{Claim}
\newtheorem{proposition}[theorem]{Proposition}
\newtheorem{definition}[theorem]{Definition}
\newtheorem{example}[theorem]{Example}
\newtheorem{remark}[theorem]{Remark}

\newcommand{\field}[1]{\mathbb{#1}}
\newcommand{\C}{\field{C}}
\newcommand{\K}{\field{K}}
\newcommand{\N}{\field{N}}

\newcommand{\R}{\field{R}}

\newcommand{\wh}[1]{\widehat{#1}}
\newcommand{\wt}[1]{\widetilde{#1}}
\newcommand{\aph}{{\alpha}}
\newcommand{\bun}{{\bf 1}}
\newcommand{\ba}{{\bf a}}
\newcommand{\bA}{{\bf A}}
\newcommand{\bb}{{\bf b}}
\newcommand{\bbl}{{\bf bl}}
\newcommand{\bbo}{{\bf 0}}
\newcommand{\bB}{{\bf B}}

\newcommand{\bc}{{\bf c}}
\newcommand{\bd}{{\bf d}}
\newcommand{\be}{{\bf e}}
\newcommand{\bbf}{{\bf f}}

\newcommand{\bg}{{\bf g}}
\newcommand{\bG}{{\bf G}}
\newcommand{\bh}{{\bf h}}
\newcommand{\bH}{{\bf H}}

\newcommand{\bm}{{\bf m}}
\newcommand{\bM}{{\bf M}}
\newcommand{\bMbar}{\overline{\bf M}}

\newcommand{\bP}{{\bf P}}

\newcommand{\bs}{{\bf s}}
\newcommand{\bS}{{\bf S}}
\newcommand{\bu}{{\bf u}}

\newcommand{\bx}{{\bf x}}
\newcommand{\by}{{\bf y}}
\newcommand{\bz}{{\bf z}}

\newcommand{\cD}{\mathcal{D}}

\newcommand{\cL}{\mathcal{L}}
\newcommand{\Cn}{{\C^n}}

\newcommand{\cO}{\mathcal{O}}
\newcommand{\cU}{\mathcal{U}}
\newcommand{\cV}{\mathcal{V}}
\newcommand{\cX}{\mathcal{X}}

\newcommand{\clos}{{\rm clos}}
\newcommand{\crit}{{\rm crit}}
\newcommand{\dd}{{\partial}}
\newcommand{\dist}{{\rm dist}}
\newcommand{\dlt}{{\delta}}

\newcommand{\ff}{{\rm ff}}
\newcommand{\gR}{{\geq R}}
\newcommand{\gm}{{\gamma}}
\newcommand{\ini}{{\rm in}}
\newcommand{\inn}{{{\rm in}}}
\newcommand{\Kn}{{\K^n}}
\newcommand{\Kp}{{\K^p}}

\newcommand{\lbd}{{\lambda}}

\newcommand{\lgth}{{\rm length}}
\newcommand{\lR}{{\leq R}}
\newcommand{\mfrk}{\mathfrak{m}}
\newcommand{\nb}{{\nabla}}
\newcommand{\omg}{\omega}
\newcommand{\Omg}{\Omega}

\newcommand{\ovS}{{\overline{S}}}
\newcommand{\ovX}{{\overline{X}}}
\newcommand{\rd}{{\rm d}}
\newcommand{\Rgo}{{\R_{\geq 0}}}
\newcommand{\Rn}{{\R^n}}
\newcommand{\Rnbar}{\overline{\R^n}}
\newcommand{\sing}{{\rm sing}}
\newcommand{\sgm}{\sigma}
\newcommand{\Sgm}{\Sigma}
\newcommand{\Ufrk}{\mathfrak{U}}
\newcommand{\ups}{{\upsilon}}

\newcommand{\Zbar}{\overline{Z}}
\newcommand{\zt}{{\zeta}}
\numberwithin{equation}{section}

\numberwithin{equation}{section}

\begin{document}
%%%%%%%%%%%%%%%%%%%%%%%%%%%%%%%%%%%%%%%%%%%%%%%%%%%%%%%%%%%%%%%%%%%%%%
%%%%%%%%%%%%%%%%%%        Titre     %%%%%%%%%%%%%%%%%%%%%%%%%%%%%%%%
%%%%%%%%%%%%%%%%%%%%%%%%%%%%%%%%%%%%%%%%%%%%%%%%%%%%%%%%%%%%%%%%%%%%%
\title[conic singular sub-manifolds and algebraic sets]{Global Lipschitz geometry 
of conic singular sub-manifolds with applications to algebraic sets.}

%    Information for first author
\author[A. Costa]{Andr\'e Costa}
%\email{}
%    \thanks will become a 1st page footnote.
\thanks{The second named author is very grateful to Daniel Cibotaru for 
motivating discussions and comments.}
%
%    Information for second author
\author[V. Grandjean]{Vincent Grandjean}
% \email{vgrandjean@math.ufc.br}
%\thanks{Support information for the second author.}

%    Information for third author
\author[M. Michalska]{Maria Michalska}
\address{
A. Costa, 
Centro de Ci\^encias e tecnologias, 
Universidade Estadual do Cear\'a, 
Campus do Itaperi, 60.714-903 Fortaleza - CE, Brasil}
\address{V. Grandjean, 
Departamento de Matem\'atica, Departamento de Matem\'atica,
Universidade Federal de Santa Catarina, 
88.040-900 Florianópolis - SC, Brasil,
Brasil}
\address{M. Michalska, Wydzia\l{} Matematyki i Informatyki, Uniwersytet 
\L{}\'o{}dzki, Banacha 22, 90-238 \L{}\'o{}d\'z{}, Poland}

\email{
andrecosta.math@gmail.com - 
vincent.grandjean@ufsc.br
 - maria.michalska@wmii.uni.lodz.pl
}

%    General info
\subjclass[2000]{}

%\footnote{\textbf{VERSION PROVISOIRE DU 15 MAI 2004}}

%\dedicatory{This paper is dedicated to our authors.}

% \keywords{}
% _X
% \date{\today}

\begin{abstract}
The main result states that a connected conic singular sub-manifold 
of a Riemannian manifold, compact when the ambient manifold is non-Euclidean, 
is Lipschitz Normally Embedded: the outer and inner metric space structures are 
metrically equivalent. We also show that a closed subset of $\Rn$
is a
conic singular sub-manifold if and only if its closure in the one point 
compactification $\bS^n =\Rn\cup \infty$ is a conic singular
sub-manifold. Consequently the connected components of generic
affine real and complex algebraic sets are conic at infinity, thus are 
Lipschitz Normally Embedded. 
\end{abstract}

\maketitle
\setcounter{tocdepth}{1}
\tableofcontents
%%%%%%%%%%%%%%%%%%%%%%%%%%%%%%%%%%%%%%%%%%%%%%%%%%%%%%%%%%%%%%%%%%%%%%%%%%
%%%%%%%%%%%%%%%%%%%%%%%%%%%%%%%%%%%%%%%%%%

%%%%%%%%%%%%%%%%%%%%%%%%%%%%%%%%%%%%%%%%%%%%%%%%%%%%%%%%%%%%%%%%%%%%%%%%%%
%%%%%%%%%%%%%%%%%%%%%%%%%%%%%%%%%%%%%%%%%%

\section{Introduction}
$ $

A subset $S$ of a smooth Riemannian manifold $M$ is called \em quasi-convex, Whitney 1-regular 
\em or \em Lipschitz normally embedded \em (abbreviated to LNE) if there exists a constant $L$ such that
\begin{equation}\label{eq:prop-LNE}
d_\inn^S \; \leq L \; \cdot d_S
\end{equation}
where $(S,d_S)$ is the outer metric structure, i.e. the distance in $S$
is taken in the ambient space $M$, and $(S,d_\inn^S)$ is the inner metric structure, where the
distance between two points of $S$ is taken as the infimum of the length of the 
rectifiable curves lying in $S$ connecting the two points. This notion was well-established first by Whitney in \cite{Whi1,Whi2}, thereafter was studied in \cite{Sta,Kur1} for sub-analytic sets; subsequently was re-introduced under the name of quasi-convex sets in the investigation of length spaces, see \cite{Gro1}. The least ambivalent name of Lipschitz normally embedded sets, therefore the one we will use, was introduced by \cite{BiMo} in the semi-algebraic context.

The present paper continues to apply geometric results on LNE 
subsets of Riemannian manifolds to global Lipschitz geometry of algebraic 
sets, which started in the PhD thesis of the first author \cite{Cos} and in 
our joint articles \cite{CoGrMi1,CoGrMi2}. A classical type of singularities 
in a Riemannian setting are the metrically conically singular ones, 
intensively studied since the seminal works \cite{Che1,Che2,ChTa1,ChTa2}. In 
this paper we research their local and global LNE nature. 
Investigating Lipschitz properties in a unbounded metric space would usually 
involve studying the behaviour along the geodesic rays
\cite{Gro2,Gro1}, whereas the methods of our paper \cite{CoGrMi2} allow us to reduce 
the study of the global LNE property of conic singular sub-manifolds of $\Rn$ to their 
 one-point compactification in $\bS^n$.
On another hand, there is a growing body of work on LNE algebraic and analytic subset germs 
and one might consult \cite{FaPi,Pic,MeSa} to overview the state of the art 
at this time. The local LNE problem at infinity 
was initiated in \cite{FeSa} for complex algebraic sets. 
Besides the obvious examples of global LNE sets that are the 
compact connected sub-manifolds and $\K$-cones over these, prior to 
\cite{Cos} and \cite{CoGrMi1} completely characterizing LNE complex algebraic curves of $\Cn$, 
only a single paper  \cite{KePeRu} presented 
non-trivial examples of globally LNE algebraics sets of $\Kn$. A complete bi-Lipschitz invariant of the outer metric space structure of complex affine plane curves is built in \cite{Tar}, yet does not describe explicitly the LNE ones.

\medskip
This paper first establishes results for conic singular sub-manifolds, i.e. 
subsets of a manifold $M$ with (metrically) conic singular points, see page
\pageref{def:conic-point}. 
  
\medskip\noindent
{\bf Theorem \ref{thm:main-compact}.}  \em
A compact connected conic singular sub-manifold of a Riemannian manifold is LNE.
\em
 
\medskip
After introducing the notion of conic at infinity in Section 4, 
%%and by methods similar to those of \cite{CoGrMi2}, 
%%(see \cite{Sam} as well), 
we also obtain the following

\medskip\noindent
{\bf Theorem \ref{thm:main-Rn}.}
\em A connected conic singular sub-manifold of $\Rn$ is LNE. \em

\medskip
Note that Theorems \ref{thm:main-compact} and 
\ref{thm:main-Rn} below would still hold true 
if we were to work with $C^2$ conic singular sub-manifolds instead of smooth 
ones,  after the necessary adjustments.

\medskip
Our next step is to apply these results to study LNE property 
among algebraic subsets of $\Kn$, where $\K=\C$ or $\R$. We prove

\medskip\noindent
{\bf Theorem.} \em
Generic algebraic sets of $\Kn$ are conic singular sub-manifolds of $\Kn$
(Lemma \ref{lem:transverse}, Proposition \ref{prop:generic-cut-LNE-algebraic},
Proposition \ref{prop:CI-gen-LNE}). \em

\medskip
Straightforward consequences of the latter result are 

\medskip\noindent
{\bf Corollaries.} \em
1) Connected components of generic algebraic sets of $\Kn$ are LNE.
\\
2) Connected components of regular fibres of generic polynomial 
mappings $\Kn \to \Kp$ are LNE with $p\leq n$.
\\
4) Generic complex ICIS germs are LNE at their singular point.
\em

\medskip
The paper is organised as follows. Section \ref{section:LNE} presents 
notations, definitions, and basic properties.
Section \ref{section:CSP} introduces the notion of conic 
singular point and presents the key collar neighbourhood Lemma 
\ref{lem:collar-origin}. The main result, Theorem \ref{thm:main-compact} 
established in Section \ref{section:MRCC} is mostly a consequence of 
Claim \ref{claim:graph-origin} obtained from the collar neighbourhood lemma. 
Section \ref{section:CL} deals with the Euclidean situation to obtain Theorem 
\ref{thm:main-Rn} and Proposition \ref{prop:conic-spheric}.
Section \ref{section:PPoV} deals with general affine traces of projective algebraic sets, 
while Section \ref{section:APoV} presents two genericity LNE results about affine algebraic sets. Coincidentally,
the material developed so far adapts almost readily to yield Proposition 
\ref{prop:ICIS-germ-LNE}, a local LNE result for generic complex ICIS 
germs. Last, Appendix \ref{section:appendix} recalls, with proofs, 
some genericity results about polynomial mappings with bounded degrees,
used in Section \ref{section:APoV} and Section 
\ref{section:SAtICISG}. 
%
%
%
%
%
%
%
%
%
%
%
%
%
%
%
%
%
%
%
%
%
%
%
%
%
%
%       *****************************************************************
%
%
%
%
%
%
%
%
%
%
%
%
%
%
%
%
%
%
%

\section{Preliminaries}\label{section:LNE}
Throughout the paper smooth means $C^\infty$.
\subsection{Miscellaneous}
$ $ 

The Euclidean space $\Rn$ is equipped with the Euclidean distance, denoted
$|-|$. We denote by $B_r^n$ the open ball of $\Rn$ of radius $r$ and centred 
at the origin $\bbo$, by $\bB_r^n$ its closure and by $\bS_r^{n-1}$ its 
boundary. The open ball of radius $r$ and centre $\bx_0$ is $B^n(\bx_0,r)$, 
its closure is $\bB^n(\bx_0,r)$ and $\bS^{n-1}(\bx_0,r)$ is its boundary.
The unit sphere of $\Rn$ centred at the origin is $\bS^{n-1}$.

We recall that the half-plane $[0,\infty) \times \R^{n-1}$ inherits the 
topology of $\Rn$. Let $U$ be an open subset $[0,\infty) \times \R^{n-1}$.
A smooth function $f:U\to \R$ is the restriction to $U$ of a smooth function 
$F:\cU\to\R$ where $\cU $ is an open subset of $\Rn$ containing $U$.

\smallskip
A topological manifold 
with boundary $(M,\dd M)$ is  a second countable Hausdorff space that admits an atlas of charts
$(U,\phi)$ of 
$(M,\dd M)$ with domain (the open subset) $U$ such that it is a homeomorphism $\phi:U 
\to V$ into the open subset $V$ of $[0,\infty) \times \R^{n-1}$.
For the manifold to be smooth we assume that for any pair of charts $(U,\phi)$ and $(U',\phi')$ both 
transition mappings $\phi\circ(\phi')^{-1}$ and $\phi'\circ\phi^{-1}$ are 
smooth, with the convention that a map defined over the empty set is smooth.

Let $X$ be a subset of $\Rn$ containing $\ba$. Let $?$ be any element of $\{<,\leq,>,\geq\}$. We denote
\begin{equation}\label{eq:Xar}
X(\ba)_r := X\cap \bS^{n-1}(\ba,r), \;\; {\rm and} \;\; 
X(\ba)_{? \, r} := X \cap \{|\bx-\ba| \, ? \, r\}, \;\;
\end{equation}

\subsection{Lipschitz normally embedded sets}
$ $

Let $(M,g_M)$ be a smooth Riemannian manifold. The metric 
tensor $g_M$ induces the 
distance $d_M$ on $M$ defined as the infimum of the lengths of rectifiable 
curves connecting any given pair of points. 
Thus, any subset $S$ of $M$ admits two natural metric space structures 
inherited from $(M,g_M)$:
\begin{definition}\label{def:outer-inner}
	1) The \em outer metric space structure $(S,d_S)$, \em is equipped with 
	the \em outer distance \em function $d_S$, 
	restriction of $d_M$ to $S\times S$. 
	
	\smallskip\noindent
	2) The \em inner metric space structure $(S,d_\inn^S)$, \em is equipped with
	the \em inner distance \em function 
	$$
	d_\inn^S: S\times S \to [0,+\infty]
	$$
	defined as follows: given $\bx,\bx'\in S$, the number $d_\inn^S(\bx,\bx')$ is the infimum of the lengths
	of the rectifiable paths lying in $S$ joining $\bx$ and $\bx'$.   
\end{definition}
Observe that $d_S \leq d_\inn^S$. 

\begin{definition}\label{def:LNE}
	Let $(M,g_M)$ be a smooth Riemannian manifold. 
	
	\smallskip
	i) A subset $S$ of $(M,g_M)$ is \em Lipschitz normally 
	embedded \em (shortened to LNE) if 
	%%the identity mapping $(S,d_S)  \to (S,d_\inn^S)$ is bi-Lipschitz:
	there exists a positive constant $L$ such 
	that
	$$
	\bx,\bx' \; \in \; S \; \Longrightarrow \; d_\inn^S(\bx,\bx') \; \leq \;
	L \cdot d_S(\bx,\bx').
	$$
	Any constant $L$ satisfying the previous inequality is called a \em LNE 
	constant of $S$. \em 
	
	\smallskip
	(ii) The subset $S$ of $(M,g_M)$ is \em locally LNE at $\bx_0$ \em if 
	there exists a neighbourhood $U$ of $\bx_0$ in $M$ such that
	$S\cap U$ is LNE.
	
	\smallskip
	(iii) The subset $S$ of $(M,g_M)$  is \em locally LNE \em if it is locally LNE 
	at each of its points.
\end{definition}
\begin{remark}
	If the manifold $M$ is compact, the property of being LNE
	depends only on the $C^1$ structure of $M$.
\end{remark}
\noindent
We will simply use the expression LNE, but it is always understood 
that it is with respect to the metric space structure given by the Riemannian 
metric $g_M$ on $M$. 
The next result characterizes completely compact LNE subsets
(see \cite[Proposition 2.1]{BiMo} without proof, \cite[Proposition 2.4]{KePeRu} and \cite[Lemma 2.6]{CoGrMi1}).
\begin{proposition}\label{lem:compact-LNE}
	A connected compact subset of the smooth Riemannian manifold $(M,g_M)$ is 
	LNE if and only if it is locally LNE at each of its points. In particular 
	any connected compact $C^1$ sub-manifold, possibly with boundary, is LNE.
\end{proposition}
%
%
%%We list a few examples and counter-examples. 
%%%
%%%
%%\begin{example}
%%0) Convex sets are $LNE$.
%%\\
%%1) Any compact connected $C^1$ sub-manifold of $\Rn$ is LNE, see Lemma 
%%\ref{lem:compact-LNE} and \cite{CoGrMi1}.
%%\\
%%2) The non-singular curve $\{y - x^2 = 0\}$ of $\K^2$ is not LNE
%%(see \cite{FeSa,CoGrMi1}).
%%\\
%%3) The connected components of the non-singular curve $\{x^2 - y^2 = 1\}$ of 
%%$\K^2$ is LNE (see \cite{CoGrMi1}). 
%%\\
%%4) Let $S$ be a closed connected $C^1$ sub-manifold of $\bS^{n-1}$, 
%%the non-negative cone $\wh{S}^+$ over $S$ 
%%is LNE, see Lemma \ref{lem:cone-LNE} and \cite[Proposition 2.8]{KePeRu}.
%%\\
%%5) As a variation of 4) we have the following: Let $X$ be a non-singular 
%%projective variety of $\bH_\infty$, the hyperplane 
%%at infinity of $\field{KP}^n$. The $\K$-cone over $X$ defined as
%%$$
%%\widehat{X} := \cup_{\lambda\in X}\K\lambda \; \subset \; \K^n
%%$$
%%is also LNE, where $\K\lambda$ is the $\K$-vector line with direction 
%%$\lambda$. 
%%\\
%%6) Given a positive integer $d$ and a non-zero constant $c$, 
%%the connected components of the non-singular surface of 
%%$x^d + y^d = c + z^d$ of $\K^3$ is LNE, see our 
%%main result Theorem 
%%\ref{thm:main-Rn}. 
%%\end{example}
%%%
%%%
%%A proof of Example 5) adapts from one of Example 4), which is 

\subsection{Non-negative cones}

$  $

Let $\Rgo := [0,\infty)$. The half-line in the oriented direction (of the 
vector) $\bu\in\Rn\setminus \bbo$ is 
$$
\Rgo \bu := \{t\bu : t \in \Rgo\}.
$$
The non-negative cone over the subset $S$ of $\Rn$ with vertex $\bx_0$ is 
defined as 
$$
\wh{S}^+ := \bx_0 + \cup_{\bx\in S\setminus\bx_0}\,\Rgo(\bx-\bx_0).
$$ 

A special case of interest is the following result (see \cite[Proposition 
2.8]{KePeRu})
\begin{lemma}\label{lem:cone-LNE}
	A non-negative cone over an LNE subset of the unit sphere  $\bS^{n-1}$ is LNE. 	
	
	%More precisely, let $S$ be a closed connected subset of $\Rn$ contained in $\bS^{n-1}$. 
	%If $S$ is LNE, then the non-negative cone $\wh{S}^+$ is LNE. In particular,
	%when $S$ is a compact connected $C^1$ sub-manifold of $\bS^{n-1}$, the cone
	%$\wh{S}^+$ is LNE. 
\end{lemma}
\begin{proof}
	Let  $S$ be a closed connected subset of $\Rn$ contained in $\bS^{n-1}$ and $L$ be a LNE constant of $S$.
	
	Let $\bx,\bx'$ points of $\wh{S}^+$. We can assume that the segment 
	$[\bx,\bx']$ is not contained in $\wh{S}^+$. Therefore there exist
	$x,x'>0$ and $\bu,\bu' \in \bS^{n-1}$ such that
	$$
	\bx = x \bu \;\; {\rm and} \;\;  \bx' = x' \bu'.
	$$
	Thus $\bu$ and $\bu'$ are different. Let $\gm$ be any 
	rectifiable path of $S$ connecting $\bu$ and $\bu'$ and such that
	$$
	\lgth(\gm) \leq 2L\cdot |\bu-\bu'|.
	$$
	If $x =x'$, let $x\gm$ be the path in $\wh{S}^+$ defined as 
	$t \mapsto x\cdot \gm(t)$. Thus 
	$$
	\lgth(x\gm) \leq 2 L \cdot |\bx - \bx'|.
	$$
	Assume $x' > x$. 
%%	We define the following additional point 
%%	$$
%%%%	\by' := x' \bu \;\; {\rm and} \; \; 
%%\by := x \bu'.
%%	$$
	Let us write $d_\inn$ for $d_\inn^{\wh{S}^+}$. We deduce that
	$$
	d_\inn (\bx,\bx') \leq d_\inn(\bx,x\bu') + d_\inn (x\bu',\bx') \leq
	\lgth(x\gm) + (x'-x) \leq (2L + 1) |\bx - \bx'|.
	$$
%%	
%%	If $S$ is a connected closed $C^1$ sub-manifold of $\bS^{n-1}$, it is LNE by 
%%Proposition \ref{lem:compact-LNE}.
\end{proof}
%
%
%The first consequence of Lemma \ref{lem:cone-LNE} is the following
%slight generalisation
%
%
As an application of the law of cosines and of Lemma \ref{lem:cone-LNE}, the following 
hold true 
\begin{corollary}\label{cor:union-cone-LNE}
	Let $S$ be a closed subset of $\bS^{n-1}$ such that each ot its 
	connected component  $S_1, \ldots,S_c,$ is LNE. 
    Then the non-negative cone $\cup_{i=1}^c \wh{S_i}^+$ is LNE.
\end{corollary}
%
%
%\begin{proof}
%Let $C := \cup_{i=1}^c \wh{S_i}^+$.
%Since the closed subsets $S_i$ are disjoint, we get
%$$
%\dlt := \min\{\dist(S_i,S_j) : 1\leq i < j \leq c\} >0,
%$$
%where $\dist(-,-)$ is the distance in $\Rn$. 
%Let $\bx_k = r_k \bu_k \in \wh{S_k}^+ \setminus \bbo$ for $k=i,j$ with 
%$\bu_k\in S_k$. Let $2\aph \in [0,\pi]$ be the non-oriented angle between 
%$\bu_i$ and $\bu_j$. Therefore we get
%$$
%|\bu_i - \bu_j| = 2\sin\aph \geq \dlt.
%$$
%Since the shortest path to connect $\bx_i$ and $\bx_j$ in $C$ is the union of 
%the two segments $[\bx_i,\bbo]$ and $[\bbo,\bx_j]$, we find the following 
%inner distance estimate 
%$$
%d_\inn^C(\bx_i,\bx_j) = r_i + r_j.
%$$
%We recall the following formula
%%
%%
%\begin{equation}\label{eq:bxi-bxj}
%|\bx_i - \bx_j|^2 = (r_i-r_j)^2 \cos^\aph + (r_i+r_j)^2\sin^2\aph,
%\end{equation}
%%
%%
%From which we deduce 
%$$
%d_\inn^C (\bx_i,\bx_j) \leq \frac{2}{\dlt}|\bx_i 
%-\bx_j|.
%$$
%\end{proof}
%%
%
%The next consequences of Lemma \ref{lem:cone-LNE} are straightforward.
%
%
\begin{corollary}\label{cor:cone-LNE}
	1) Let $S$ be a closed connected subset of $\Rn$ contained in $\bS^{n-1}$.
	If $S$ is LNE, then the following subsets are also LNE
	$$
	\wh{S}^+ \cap B_R^n, \;\; \wh{S}^+ \cap \bB_R^n, \;\; 
	\wh{S}^+ \setminus B_R^n, \;\; \wh{S}^+ \setminus \bB_R^n, \;\;  
	{\rm and} \;\; \wh{S}^+ \cap \bS_R^{n-1}, 
	$$ 
	for any positive radius $R$.
	
	\medskip\noindent
	2) Let $S$ be a closed subset of $\bS^{n-1}$ such that the non-negative cone 
	$\wh{S}^+$ is LNE. Then for each positive radius $R$ the truncated cones 
	$\wh{S}^+\cap \bB_R^n$ and $\wh{S}^+\cap B_R^n$ are LNE. 
\end{corollary}
%
%

%
%
%
%
%
%
%
%
%
%
%
%
%
%
%
%
%
%
%
%
%
%
%
%
%
%
%
% 
%         *******************************************************
%
%
%
%
%
%
%
%
%
%
%
%
%
%
%
%
%
%
%
%

\subsection{Sub-manifolds}
$ $

Let $(M,\dd M)$ be a smooth manifold with boundary $\dd M$ (possibly empty).
A subset $N$ of $(M,\dd M)$ is a smooth sub-manifold of $(M,\dd 
M)$ with boundary $\dd N$ (possibly empty) if: (i) 
it is a smooth manifold with boundary $(N,\dd N)$; (ii) its manifold 
topology coincides with the induced topology from $M$; (iii) the 
inclusion mapping $N \hookrightarrow M$ is a smooth injective immersion.
\begin{definition}\label{def:p-sub}(\cite[Section I.7]{Mel})
	Let $(M,\dd M)$ be a smooth manifold with boundary. A subset $N$ of $M$ is a 
	\em p-sub-manifold \em if: (i) it is a smooth sub-manifold with boundary $\dd 
	N$ of $M$; (ii) $\dd N$ is contained in $\dd M$; (iii) $N$ is transverse to 
	$\dd M$:
	$$
	\bx \in \dd N \; \Longrightarrow \; T_\bx N + T_\bx \dd M = T_\bx M.
	$$
\end{definition}

\subsection{Spherical blow-up}
$ $

We use here basic notions presented in \cite{Mel} in our elementary embedded context.

The spherical blowing-up of $\Rn$ with centre the point $\ba$ of $\Rn$
is the mapping defined as
$$
\bbl_\ba :[\Rn,\ba] = \Rgo \times \bS^{n-1} \to \Rn, 
\;\; (r,\bu) \to r\bu + \ba.
$$
Its domain $[\Rn,\ba]$ is a smooth manifold with smooth compact 
boundary 
$\dd[\Rn,\ba] = 0\times \bS^{n-1}$. The blowing-down mapping 
is a diffeomorphism outside the boundary onto $\Rn\setminus \ba$
\begin{definition}\label{def:front-face}
	The \em front face \em of $[\Rn,\ba]$ is $\ff([\Rn,\ba]) := \dd[\Rn,\ba]$.
\end{definition}
We need the following notions.
\begin{definition}\label{def:strict-transform}
	Let $S$ be a subset of $\Rn$ such that its closure contains the 
	point $\ba$. The \em strict transform of $S$ by $\bbl_\ba$
	is the subset of $[\Rn,\ba]$ defined as
	$$
	[S,\ba]:=\clos(\bbl_\ba^{-1}(S\setminus\ba)).
	$$
	The \em front face \em of the strict transform of $S$ is 
	$$
	\ff([S,\ba]) := [S,\ba] \cap \ff([\Rn,\ba]) = 0\times S_\ba S,
	$$
\end{definition}
where 
$$
S_\ba S :=\left\{\bu \in \bS^{n-1} : \exists (\bx_k)_k\in S\setminus \ba 
\; {\rm with} \; \bx_k \to \ba \; : \; \frac{\bx_k-\ba}{|\bx_k - \ba|} 
\to \bu 
\right\}.
$$
The non-negative cone $\wh{S_\ba S}^+$ over $S_\ba S$ is the 
\em tangent cone of $S$ at $\ba$. \em 

\medskip
If $S$ is a connected smooth sub-manifold of $\Rn$ containing the point 
$\ba$, its strict transform $[S,\ba]$ is a smooth sub-manifold 
with smooth compact boundary $\dd[S,\ba] = \ff([S,\ba])$. More precisely, 
in such a case the strict transform $[S,\ba]$ is a p-sub-manifold of 
$[\Rn,\ba]$.

\medskip
The following result is well known, part of the folklore, and left as
an exercise.
\begin{proposition}\label{prop:blow-up-intrinsic}
	Let $\phi : N \to N'$ be a smooth diffeomorphism between the smooth
	sub-manifolds $N$ of $\Rn$ and $N'$ of $\R^{n'}$. If $\ba$ is a point
	of $N$, then the smooth diffeomorphism 
	$$
	\bbl_{\phi(\ba)}^{-1} \circ \phi \circ \bbl_\ba : [N,\ba] \setminus
	\ff ([N,\ba]) \to [N',\phi(\ba)] \setminus \ff ([N',\phi(\ba)])
	$$
	extends as a smooth diffeomorphism $[N,\ba] \to [N',\phi(\ba)]$ between
	smooth sub-manifolds with smooth compact non-empty boundaries.
\end{proposition}

We finish this section with elementary observations.
Let $N$ be a sub-manifold of $\Rn$, and let $\ba$ be a point of $N$.
The blowing-up of $\ba$ in $N$ is the following mapping
$$
\bbl_\ba^N := \bbl_\ba|_{[N,\ba]} : [N,\ba] \to N.
$$
If $S$ is subset of $N$, its strict transform by $\bbl_\ba^N$ is
$$
[S,\ba]^N := \clos((\bbl_\ba^N)^{-1}(S\setminus \ba))
$$
and its front face is defined as
$$
\ff^N([S,\ba]^N) := [S,\ba]^N \cap \ff([N,\ba]).
$$
The following identities hold true
$$
[S,\ba]^N = [S,\ba] \;\; {\rm and} \;\; \ff^N([S,\ba]^N) = \ff([S,\ba]).
$$

\section{Conic singular points}\label{section:CSP}
We present here the principal object of our point of view. 
The interest in a metrically conical point in a Riemannian context (shorten to conic singular below) dates back at least to \cite{Che1,ChTa1,ChTa2,Che2} (see also \cite{MeWu} and \cite{Gri} for an interesting variation). Our context dictates to present a notion of conic (singular) point using charts. Proposition \ref{prop:blow-up-intrinsic} guarantees that the notion of conic point
we define is intrinsic, independent of any $C^2$ Riemannian structure. 
 Although our notion of conic singular point of a subset is differential, 
whenever the subset inherits the inner metric from the ambient Riemannian 
manifold, it turns into a metrically conical point.

\begin{definition}
A \em singular point $\ba$ of a subset $S$ \em of a smooth manifold is a point at which the subset germ $(S,\ba)$ is not that of a smooth sub-manifold. Otherwise we say that the point \em $\ba$ is a smooth point \em of $S$. The \em singular locus $S_\sing$ of $S$ \em is the set of singular
points of $S$.
\end{definition}

\begin{definition}\label{def:conic-point}
	Let $X$ be a subset of $\Rn$ such that its germ $(X,\ba)$ at $\ba \in \Rn$
	is non-empty and closed.

	The point $\ba$ is \em a conic point of $X$, \em 
	if there  exists a positive radius $r$ such that the strict transform  
	$[X(\ba)_{< r},\ba]$ is a closed subset and a p-sub-manifold of 
	$[\Rn(\ba)_{<r},\ba]$.
	
	%\smallskip
	%ii)  A \em conic singular point of $X$ \em is a conic point of $X$ at which 
	%the subset germ $(X,\ba)$ of $(\Rn,\ba)$ is not that of a smooth sub-manifold 
	%of $\Rn$
\end{definition}
	If $\ba$ is a conic point of the subset $X$ of $\Rn$, then $S_\ba X$ is a 
	smooth compact sub-manifold of $\bS^{n-1}$.  

%\bigskip
%Let $N$ a sub-manifold of $\Rn$.
%%
%%
%\begin{definition}\label{def:conic-in-submfd}
%	Assume that the subset germ $(S,\ba)$ of $(N,\ba)$ is closed and non-empty.
%	
%	
%	i) The point $\ba$ is \em a conic point of $S$, \em 
%	if the subset germ 
%	$([S,\ba]^N,\ff^N([S,\ba]^N))$ is a p-sub-manifold of 
%	$([N,\ba],\ff([N,\ba]))$.
%	
%	\smallskip
%	ii)  A \em conic singular point of $S$ \em is a conic point of $S$ at which 
%	the subset germ $(X,\ba)$  is not that of a smooth sub-manifold 
%	of $N$.
%\end{definition}
%%
%%
%The following property allows us to treat conic points of subsets of embedded manifold simply as subsets of the euclidean space. 
%%
%%
%\begin{proposition}\label{claim:conic-N-conic}
%	The subset $S$ of $N$, sub-manifold of $\Rn$, is conic in $N$ at $\ba$, 
%	if and only if the subset $S$ of $\Rn$ is conic at $\ba$. 
%\end{proposition}
%\begin{proof}
%Since The following identities hold true
%$$
%[S,\ba]^N = [S,\ba] \;\; {\rm and} \;\; \ff^N([S,\ba]^N) = \ff([S,\ba]),
%$$
% we get
% 
% \smallskip \label{def:conic-point}
% i) The point $\ba$ is \em a conic point of $S$, \em 
% if there  exists a positive radius such that the strict transform  
% $[S(\ba)_{< r},\ba]$ is a closed p-sub-manifold of 
% $[\Rn(\ba)_{<r},\ba]$.
% 
% \smallskip
% ii)  A \em conic singular point of $X$ \em is a conic point of $X$ at which 
% the subset germ $(S,\ba)$ of $(\Rn,\ba)$ is not that of a smooth sub-manifold 
% of $\Rn$.
%\end{proof}

%
%
%
%
%

%
%
\begin{definition}\label{def:conic-abstract}
	Let $M$ be a smooth manifold, $Z$ be a subset of $M$
	and $\bz$ be a point of $Z$.

	The subset $Z$ is \em conic at $\bz$, \em if there exists a smooth 
	chart $\psi: \cU \to B_1^n$ of $M$ centred at $\bz$ such that $\psi(Z)$ is conic at 
	$\psi(\bz)$.
	
	%	\smallskip
	%	ii) The subset $Z$ has \em a singular conic point at $\bz$, \em 
	%	if there exists a smooth chart there exists a smooth chart $\psi: \cU \to 
	%	B_1^n$ such that $\psi(\bz)$ is singular conic point of $\psi(Z)$.
\end{definition}

Definition \ref{def:conic-abstract} is well-posed: let $\bz\in Z$ and $\psi_i:\cU_i \to B_1^n$
be two smooth charts centred at $\bz$ such that $\psi_i(\bz) = \ba_i$. 
Let $S_i := \psi_i(Z)$. Proposition \ref{prop:blow-up-intrinsic} 
implies that $S_1$ is conic at $\ba_1$ if and only if $S_2$ is 
conic at $\ba_2$, and similarly $S_1$ is singularly conic at $\ba_1$
if and only if $S_2$ is singularly conic at $\ba_2$. 

If $\bz$ is a conic point of the subset $Z$ of the smooth manifold $M$,
then the germ $(Z\setminus \bz,\bz)$ is that of smooth sub-manifold. 
Thus any conic singular point of $Z$ is an isolated singular point of $Z$. 
Moreover, at a conic point $\bz$, the subset $Z$ admits an open neighbourhood of $\bz$ consisting only of conic points. 

The above two definitions of conic points agree for sub-manifolds of 
$\Rn$. In this elementary context the persistence of the strict transform being 
the germ of a p-sub-manifold goes through the restriction of the blowing-up to 
a containing sub-manifold, since it is a variation of the commutativity of the 
blowing-up with the base change (see also \cite[Chapter V]{Mel}).
\begin{proposition}\label{claim:conic-N-conic}
	The subset $S$ of the sub-manifold $N$ of $\Rn$, is conic in $N$ at $\ba$, if and only if the subset $S$ of $\Rn$ is conic at $\ba$. 
\end{proposition}
\begin{proof}
We have $\ba \in S \subset N \subset \Rn$, each set being locally closed at 
$\ba$, and thus $[S,\ba]^N = [S,\ba]$ and $\ff^N[S,\ba] = \ff[S,\ba] \subset 
\ff[N,\ba]$. Thus the statement.
\end{proof}
%
%{\color{red} a word on proof}

%
%
Following Definition \ref{def:conic-abstract}, we can introduce the following
\begin{definition}\label{def:conic-singular sub-manifold}
	%%Let $M$ be a smooth manifold.
	%
	A subset $W$ of a smooth manifold $M$ is a \em conic singular sub-manifold \em  if: (i) it is a closed subset of $M$, (ii) each point of $W$ is conic; (iii) the singular locus $W_\sing$ of $W$ is finite or empty. 
\end{definition}
%%
%%
%Whenever $M$ is compact and $W$ a closed subset, the singular locus 
%$\Sigma$ of the conic singular sub-manifold $W$ is at most finite. 
%	
%	A {\color{red} closed} subset $W$ of $M$ is a \em conic singular sub-manifold of $M$ \em if
%	there exists a subset $\Sigma$ of $W$, possibly empty, consisting of isolated 
%	points in $M$ such that
%	(i) $W\setminus \Sigma$ is a smooth sub-manifold of $M$; (ii) Each point of 
%	$\Sigma$ is a conic singular point of $W$.   
%
%
%
%{\color{red} Definition was corrected as the finiteness may happen to not be true. Moreover, an embedded sub-manifold needs to be closed.
	%
	%Alternative (correct) definition: A subset $W$ of a smooth manifold is a conic singular sub-manifold  if it is a closed set conic at every point and there exists a discrete set  such that at every point of $W\setminus \Sigma$ the germ of $W$ is a germ of a sub-manifold.
	%
	%
	%

%%From its definition and that of a conic point, a conic singular 
%%sub-manifold is locally closed at each of its point.
	%
	%
	\begin{remark}\label{rmk:strongly-conic}
		1) A smooth sub-manifold is conic at each of its point.
		
		\smallskip\noindent
		2) A cone $\wh{Z}^+$ over a closed smooth sub-manifold $Z$ of $\bS^{n-1}$
		has a conic point at $\bbo$. It is singular whenever $Z$ is not connected 
		or is not the intersection of $\bS^{n-1}$ with a linear sub-space of $\Rn$.
 
\smallskip\noindent
		3) If $W$ is a compact connected singular submanifold with non empty singular locus, the connected components of $W\setminus W_\sing$ may have different dimensions.   
	\end{remark}

	The following collar neighbourhood lemma is a key ingredient of the main result
%% Informally a subset conic at a point admits normal coordinates in a neighbourhood of the point 
(see also \cite[Theorem 1.2]{MeWu} for a much more general context and precise statement).
	\begin{lemma}\label{lem:collar-origin} 
		Let $X$ be a subset of $\Rn$ and $\ba$ be a point of $X$.
%%Let $\cX := [X,\ba]$ be the strict transform of $X$ by $\bbl_\ba$,
%%		and $\ff(\cX) = 0\times S_\ba X$.
				If $\ba$ is a conic point of $X$, there exists 
		$0 < r_0$ and a smooth diffeomorphism
		$$
		\tau_\ba: [0,r_0] \times S_\ba X \; \to \; [X,\ba] \cap ([0,r_0]\times \bS^{n-1}) 
		$$
		which is link-preserving, that is of the form $(r,\bu) \mapsto (r,\mu(r,\bu))$ 
		for all $r\leq r_0$. Moreover, the restriction of $\tau_\ba$ to 
		$\ff[X,\ba]$ is the identity mapping. 
	\end{lemma}
	\begin{proof}
		Without loss of generality, we can assume that $\cX: = [X,\ba]$ is closed and a p-sub-manifold with boundary $\dd \cX = \ff([X,\ba]) = 
		0 \times S_\ba X$ of $\bM =  [\Rn,\ba] = \Rgo\times\bS^{n-1}$.
		Thus $T\cX$ is a well defined vector sub-bundle of $TM|_\cX$.
		
		Let $\rd \bu^2$ be the restriction $eucl|_{\bS^{n-1}}$ of the 
		Euclidean metric tensor to $\bS^{n-1}$. Let 
		$$
		\bh = eucl|_\Rgo \otimes eucl|_{\bS^{n-1}} = \rd r^2 \otimes \rd\bu^2
		$$
		be the product metric over $\bM$.
		Let $|\cdot|_\bh$ be the norm induced by $\bh$. 
		Let $\bg$ be the Riemannian structure on $\cX$ obtained 
		from restricting $\bh$ to $\cX$ and let $|\cdot|_\bg$ be the norm induced by 
		$\bg$.
		Define the function
		$$
		r_\cX: \cX\to \R, \;\; (r,\bu) \mapsto r.
		$$ 
		Observe that it is smooth.  
		Let $\xi := \nb_\bg r_\cX$ be the gradient of $r_\cX$ taken w.r.t. the 
		metric $\bg$. 
		The vector field
		$\xi$ is smooth and decomposes in $T\bM = T\Rgo \oplus T\bS^{n-1}$ as the 
		orthogonal sum  
		$$
		\xi  = x \dd_r \oplus \xi_\bS.
		$$
		The function $x$ is smooth over $\cX$ and the 
		vector field $\xi_\bS$ is smooth over $\cX$. Since $|\dd_r|_\bh = 1$, 
		we deduce that $|\xi|_\bg^2 = x^2 + |\xi_\bS|_\bg^2 \leq 1$. 
%%        By definition of $\xi$ 		the following subsets coincide
%%		$$
%%		\{\bm\in\cX : \xi(\bm) = \bbo\} = \{\bm \in \cX : x(\bm) = 0\}.
%%		$$
		The vector field $\xi$ is smooth and tangent to $\cX$, and does not vanish
		along $\dd\cX$, since $\cX$ is transverse to $\dd\bM$. 
		There exist constants $r_1,x_1> 0$ such that 
		$$
		\min \{x(\bm) : \bm \in \cX_{\leq r_1} \} = x_1 >0, \;\; {\rm for} \;\;
\cX_{\leq r_1} := \cX \cap [0,r_1]\times \bS^{n-1}
		$$
		since $\dd\cX$ is compact. 
		The following vector field  
		$$
		\ups := \frac{1}{|\xi|_\bg^2} \,\xi
		$$
		is smooth over its domain the open subset $D := \{|\xi|_\bg>0\}$,  
		which contains $\cX_{\leq r_1}$. Let 
		$$
		\Psi: \cD \subset \R\times D \to \cX
		$$ 
		be the flow of $\ups$ for the initial Cauchy problem $\Psi(0,\bm) = \bm$.
		It is a smooth mapping over its domain $\cD$.
		Since $S_\ba X$ is compact and $\ups$ is smooth over $\cX_{\leq r_1}$, 
		there exists $r_1\geq r_0(>0)$ such  
		that $[0,r_0]\times S_\ba X$ is contained in $\cD$. 
		The smooth diffeomorphism we are looking for is 
		$$
		\tau_\ba : [0,r_0]\times S_\ba X \to \cX \cap [0,r_0]\times \bS^{n-1}, 
		\;\; (r,\bu) \mapsto \Psi(r,\bu) = (r,\mu(r,\bu)),
		$$
		since $r_\cX(\Psi(t,\bu)) = t$.
	\end{proof}
	%
	%
%%	\begin{remark}
%%1) The spherical blowing of the point  $\ba$ of a Riemannian manifold $(M,g)$ 
%%w.r.t. $g$ is well defined: $\sgm_\ba:[M,\ba] \to
%%M$. The blown-up space $[M,\ba]$ is a smooth manifold with compact boundary 
%%$\ff([M,\ba])$, which is a sphere. The pull-back metric $\sgm_\ba^*g$ is 
%%a conic metric over $([M,\ba],\ff([M,\ba])$, and its restriction to a closed 
%%p-sub-manifold $(N,\dd N)$ of $([M,\ba],\ff([M,\ba])$ yields a conic metric onto 
%%the p-sub-manifold.
%%The blowing-down image $\sgm_\ba(N)$ is possibly singular at $\ba$, in which
%%case $\ba$ is metrically conical point ($\dd N$ need not be connected).
%%\\
%%2) Although our notion of conic singular point of a subset is differential, 
%%whenever the subset inherits the inner metric from the ambient Riemannian 
%%manifold, it turns into a metrically conical point in the usual Riemannian 
%%sense (see \cite{Che1,MeWu,Gri}). 
%%%%Therefore our notion 
%%%%$(W,\ba)$ of $M$, conic at $\ba$, the pull-back of the inner metric tensor 
%%%%$\sgm_\ba^* g|_{(W\setminus\ba,\ba)}$ will extend to
%%%%$[(W,\ba),\ba]$ as a conic metric, since it is 
%%%%$(\sgm_\ba^*g)|_{[(W,\ba),\ba]}$. 
%%\end{remark}
	%
	%
	%
	%
	%
	%
	%
	%
	%
	%
	%
	%
	%
	%
	%
	%
	%
	%
	%
	%
	%
	%
	%
	%
	% 
	%         *******************************************************
	%
	%
	%
	%
	%
	%
	%
	%
	%
	%
	%
	%
	%
	%
	%
	%
	%
	%
	%
	%
	\section{Conic singular sub-manifolds are Lipschitz normally embedded}

	\subsection{Compact conic singular sub-manifolds are LNE}\label{section:MRCC}
	$ $
	
%%	{\color{red} Manifold $M$ can be taken anyhow. Then the theorem goes for any compact connected sub-manifold. Add maybe dimension of the ambient space. Manifold does not need to be complete (i.e. embedded $M$ might not be LNE itself).}
	
	Nash Embedding Theorem states 
	that any smooth Riemannian manifold $(M,g)$ embeds isometrically in
	$\Rn$ as a sub-manifold equipped with the inner metric: 
	There exists a smooth embedding
	$$
	\iota : (M,g) \to (\iota(M),eucl|_{\iota(M)}) \subset \Rn, \;\; 
	{\rm such} \; {\rm that}\;\;
	\iota^*(eucl|_{\iota(M)}) = g.
	$$
			Let $N := \iota(M)$ be equipped with its inner distance $d_\inn^N=
			(\iota^{-1})^*d_M$. Let $Y$ be a subset of $M$ and let $Z = \iota(Y)$. 
			There are four distances over $Z$: The euclidean outer distance $d_Z$; The inner distance $d_\inn^{\Rn,Z}$ on $Z$ as a subset of $\Rn$; The outer distance 
			$d_Z^N$ on $Z$ as a subset of $N$; And the inner distance $d_\inn^{N,Z}$ on $Z$ as a subset of $N$. The following estimates are obvious
		\begin{equation}\label{eq:Z-4-metrics}
		d_\inn^{N,Z} \; \geq \; \max \left(d_Z^N,d_\inn^{\Rn,Z}\right) \; \geq \; d_Z.
		\end{equation}

			By definition of Nash embedding we find
		\begin{equation}\label{eq:Z-3-metrics}
			(\iota^{-1})^* \, d_\inn^Y = \, d_\inn^{N,Z} = \, d_\inn^{\Rn,Z}.
		\end{equation}

	\medskip	
	By Claim \ref{claim:conic-N-conic} and Proposition 
	\ref{prop:blow-up-intrinsic} the following is obvious
	\begin{claim}\label{claim:conic-abstract-M-conic}
		The subset $Z$ of $M$ is conic at $\bz$ (respectively singularly conic
		at $\bz$) if and only if $\iota(Z)$ is conic in $\iota(M)$ at $\iota(\bz)$ 
		(respectively singularly conic at $\iota(\bz)$).
	\end{claim}

	The main result of the paper is
	\begin{theorem}\label{thm:main-compact}
		Let $(M,g)$ be a smooth Riemannian manifold. Then, any 
		connected, compact, conic singular sub-manifold $W$ of $M$ is LNE.
	\end{theorem}
	\begin{proof}
		By Nash Embedding Theorem, let $\iota(M,g) = (N,eucl|_N)$, for a smooth 
		sub-manifold $N$ of $\Rn$. 
		%%Without loss of generality we can assume that $N$ is connected too. 
		Following Estimates \eqref{eq:Z-4-metrics} and \eqref{eq:Z-3-metrics}
	    there is just to show that the image $X:=\iota(W)$ is LNE as a subset of $\Rn$. Claim   \ref{claim:conic-abstract-M-conic} implies that $X$ is a connected compact conic singular 
		sub-manifold of $\Rn$ with singular locus $X_\sing := \iota(W_\sing)$.
		
		If $X_\sing$ is empty, the result follows from Lemma \ref{lem:compact-LNE}.
		
		\medskip
		Assume that $X_\sing = \{\ba_1,\ldots,\ba_s\}$ with $s\geq 1$.
		
		\medskip
		Assume that $\ba_1 = \bbo$. Given a positive radius $r$, we recall that
		$$
		X_r := X \cap \bS_r^{n-1} \;\; {\rm and} \;\; X_{\leq r} := X\cap\bB_r^n.
		$$
		Let $r_0>0$ and $\tau_\bbo$ be as in Lemma \ref{lem:collar-origin}. Let $\beta = \bbl_\bbo$.
		The mapping $\beta \circ \tau_\bbo \circ \beta^{-1}$, which is a smooth 
		diffeomorphism $(\wh{S_\bbo X}^+){\leq r_0} \setminus \bbo \to
		X_{\leq r_0} \setminus \bbo$ extends as a homeomorphism 
		$$
		\phi_\bbo: (\wh{S_\bbo X}^+)_{\leq r_0} \to X_{\leq r_0}
		$$
		by continuity at $\bbo$, taking the value $\bbo$. By definition
		$\phi_\bbo$ preserves the intersection with the spheres $\bS_r^{n-1}$.
		The following property is the other key ingredient of the main result.
		\begin{claim}\label{claim:graph-origin}
			The mapping $\phi_\bbo$ is bi-Lipschitz when the source and target spaces are 
			equipped with
			their respective outer metric space structures. 
		\end{claim}
		\begin{proof}
			Let $\bM:=[\Rn,\bbo]$. If $S$ is a subset of $\bM$ and $r$ is a positive
			radius, we define the following subsets of $\bM$
			$$
			S_r := S \cap r\times\bS^{n-1} \;\; {\rm and} \;\;
			S_{\leq r} := S \cap [0,r] \times \bS^{n-1}.
			$$
			Let $\cX := [X,\bbo]$ be the strict transform of $X$ by $\beta$.
			To ease notations in the proof, write 
			$$
			C_\bbo := \wh{S_\bbo X}^+.
			$$
			
			\medskip\noindent
%%			The expression of $\phi_\bbo$ is 
%%			$$
%%			\phi_\bbo (\bx) = \left(r,r\mu\left(r,\frac{\bx}{r}\right)\right) =: (r,\vp(\bx))\;\; {\rm with}
%%			\;\; r = |\bx|.
%%			$$
%%			Let $\bx = r\bu$. For $i=1,\ldots,n,$ we find that 
%%			$$
%%			\dd_{x_i} \vp(\bx) = u_i \mu(r,\bu) + r \left( u_i\dd_r\mu (r,\bu) + u_i \dd_{u_i}\mu(r,\bu) - \frac{1}{r}\sum_{j=1}^nu_iu_j \dd_{u_j}\mu(r,\bu).
%%			\right)
%%			$$ 
%%			Thus $\dd_{x_i}\phi_\bbo$ is bounded over $(C_\bbo)_{\leq r_0}\setminus \bbo$. Thus it Lipschitz over each connected component $(C_\bbo)_{\leq r_0}\setminus \bbo$.
			Let $\bm = (r,\bu) \in \bM$ with $r>0$. 
			Then 
			$$
			T_\bm \bM = \R\times T_{\bu}\bS^{n-1}.
			$$
			Let $\bx = r\bu$, then $T_\bx \Rn$ decomposes as the orthogonal sum
			$$
			T_\bx \Rn = \R\dd_r \oplus T_{r\bu}\bS_r^{n-1}, \;\; {\rm where} \;\;
			\dd_r(\bx) := \bu.
			$$
			As vector sub-spaces of
			$\Rn$, note that $T_\bu \bS^{n-1} = 
			T_{t\bu}\bS_t^{n-1}$. 
			Using these tangent space decompositions we find 
			$$
			D_\bm \beta = \bun_\R \oplus r \cdot {\rm Id}_{T_\bu\bS^{n-1}} 
			\;\; {\rm and} \;\;
			D_\bx \beta^{-1} = \left(\bun_\R^{-1} \, , \, \frac{1}{|\bx|} \cdot 
			{\rm Id}_{T_\bu\bS^{n-1}}\right)
			$$
			where $\bun_\R : \R \to \R\bx, \;\; t \mapsto t\cdot \dd_r$. 
			Writing $\tau:=\tau_\bbo$ of Lemma \ref{lem:collar-origin}, we recall that 
			$$
			\tau (r,\bu) = (r,\mu(r,\bu)) \in \bM
			$$ 
			where $\mu : (\beta^{-1}(C_\bbo))_{\leq r_0} =  [0,r_0]\times \bS_\bbo X 
			\to \bS^{n-1}$ is a smooth submersion. The manifold $\bM$
			is equipped with the Riemannian metric $\bh = eucl|_\Rgo \otimes 
			eucl|_{\bS^{n-1}}$. Therefore $D \mu$ takes values in $T\bS^{n-1}$ and its norm is 
			uniformly bounded over the compact cylinder 
			$(\beta^{-1}(C_\bbo))_{\leq r_0}$: There exists a positive constant $L$
			such that
			$$
			\|D_\bm \mu\|_\bh \leq L, \;\; \forall \; \bm \in 
			(\beta^{-1}(C_\bbo))_{\leq r_0}.
			$$
			Let 
			$$
			\xi = (t,\xi_\bS) \in T_\bx C_\bbo = \R\times T_\bu S_\bbo X,
			$$
			with $\xi_\bS \in T_\bu \bS^{n-1}$. Thus we get 
			$$
			D_\bm \tau \cdot\xi = (t,\zt_\bS) \in \R \times T_{\mu(r,\bu)}\bS^{n-1}
			$$
			where $\zt_\bS := D_\bm \mu \cdot \xi$. Note that 
			$$
			|\zt_\bS|^2\leq L^2 \cdot (t^2 + |\xi_\bS^2|). 
			$$
			If $\vec{v} := t\dd_r \oplus \ups \in T_\bx C_\bbo = \R \times T_\bu S_\bbo X$
			with $\ups \in T_\bu S_\bbo X$ we find 
			$$
			D_\bx \phi_\bbo \cdot \vec{v} = t\dd_r \oplus D_\bm \mu \cdot 
			\left(|\bx| t, \ups \right)  .
			$$
			Therefore the norm of $D\phi_\bbo$ is uniformly bounded over $(C_\bbo)_{\leq r_0} 
			\setminus \bbo$ by $1+L$.
			The same type of arguments show that the norm of 
			$D(\phi_\bbo)^{-1}$ is also uniformly bounded over $X_{\leq r_0}\setminus \bbo$.
			Thus $\phi_\bbo$ is bi-Lipschitz over the closure of each connected 
			component of $(C_\bbo)_{\leq r_0} \setminus \bbo$.
			
			\medskip
			If $S_\bbo X$ has $c$ connected components $S_1,\ldots,S_c$. Consider
			the following non-negative cone
			$$
			C_i := \wh{S_i}^+
			$$
			so that $C_\bbo = \cup_{i=1}^c C_i$. Let 
			$$
			\dlt := \min\{\dist(S_i,S_j) : 1\leq i < j\leq c\}.
			$$
			where $\dist$ is taken in $\Rn$. Since the $S_i$ are connected closed
			smooth sub-manifold of $\bS^{n-1}$, there are pairwise disjoint, that is 
			$$
			\dlt := \min\{\dist(S_i,S_j) : 1\leq i < j\leq c\} \; > 0.
			$$
%%			To conclude that $\phi_\bbo$ is Lipschitz, we recycle the argument of the 
%%			proof of Corollary \ref{cor:union-cone-LNE}. 
			Let $\bs_k \in S_k$ and $r_k \leq r_0$ for $k=i,j$ with $i<j$.
			%%From Estimate \eqref{eq:bxi-bxj}, 
The law of cosines yields 
			$$
			\dlt (r_i + r_j) \; \leq \; |r_i\bs_i - r_j\bs_j| \; \leq \; (r_i + r_j).
			$$
			Therefore 
			$$
			|\phi_\bbo (r_i\bs_i) - \phi_\bbo (r_j\bs_j)| \leq |\phi_\bbo (r_i\bs_i)| + |\phi_\bbo (r_j\bs_j)| = r_i + r_j
			\leq \frac{1}{\dlt} |r_i\bs_i - r_j\bs_j|.
			$$
			Thus $\phi_\bbo$ is Lipschitz.
			Since $\phi_\bbo$ extends as a homeomorphism at $\bbo$, we define
			$$
			X_i := \phi_\bbo (C_i)_{\leq r_0}.
			$$
			\medskip
			Since $S_\bbo X = S_i$, we can assume that $r_0$ is small 
			enough so that 
			$$
			\min\left\{\dist\left(\frac{\bx_i}{|\bx_i|}, \frac{\bx_j}{|\bx_j|} \right), 
			\; \bx_i \in X_i, \; \bx_j \in X_j, \; 
			1 \leq  i < j \leq c \right\} \; \geq \; \frac{\dlt}{2}.
			$$
			Similarly to proving that $\phi_\bbo$ is Lipschitz, this later estimate 
			allows to obtain that the inverse mapping $(\phi_\bbo)^{-1}$ is also Lipschitz.
		\end{proof}
\noindent
		Combining  Claim \ref{claim:graph-origin} with Corollary \ref{cor:cone-LNE} 
		for each radius $r\leq r_0$, the subset
		$$
		X\cap B^n(\ba_i,r)
		$$ 
		is LNE for each $i=1,\ldots,s$. In other word $X$ is locally LNE at each
		of its singular point. Since it is locally LNE at each point of
		$X\setminus X_\sing$, at which it is a sub-manifold, the theorem is proved
		using Lemma \ref{lem:compact-LNE}.
	\end{proof}
	\begin{remark}
		Theorem \ref{thm:main-compact} extends the case of complex curves of a compact complex manifold which are LNE \cite[Section 4]{CoGrMi1}.
	\end{remark}
We leave the next corollary as an exercise.
\begin{corollary}\label{cor:conic-epointed}
Let $W$ be a compact connected singular sub-manifold of a smooth manifold
$M$. Let $\Sigma$ be a finite set of $W$. If $W\setminus \Sigma$ is connected, then 
it is LNE.
\end{corollary}

	%
	%
	%
	%
	%
	%
	%
	%
	%
	%
	%
	%
	%
	%
	%
	%
	%
	%
	%
	%
	%
	%
	%
	%
	%
	%
	%
	%
	%
	%
	%
	% 
	%         *******************************************************
	%
	%
	%
	%
	%
	%
	%
	%
	%
	%
	%
	%
	%
	%
	%
	%
	%
	%
	%
	%
	\subsection{Conic singular sub-manifolds of the Euclidean space}\label{section:CL}
$ $

As illustrated in the next example, an unbounded sub-manifold 
may not be LNE. Therefore some condition at infinity is necessary for a subset of 
$\Rn$ to be 
LNE. 
	\begin{example}\label{exparabola}
		The plane parabola $y=x^2$ (real and complex) is not LNE (see for instance \cite{FeSa,Cos,CoGrMi1}).
	\end{example}
\noindent
 We show here that connected unbounded conic singular sub-manifolds and closed in $\Rn$ are LNE once they are also conic at infinity.
	
\bigskip
	The spherical blowing-up at infinity is the following semi-algebraic 
	smooth diffeomorphism 
	$$
	\beta : \bM := \R_{>0}\times\bS^{n-1}\;  \to \; \Rn\setminus \bbo, \;\;
	(r,\bu) \mapsto \frac{\bu}{r}. 
	$$
    Let $\bMbar := \Rgo\times \bS^{n-1}$.
	We compactify $\Rn$ as the smooth manifold with boundary
	$$
	(\Rnbar,\bM^\infty) := (\bMbar \sqcup \bbo,0\times\bS^{n-1})
	$$
	using the spherical blowing-up mapping at infinity to identify $\bM$ and 
	$\Rn\setminus\bbo$. Therefore $\Rnbar$ is covered by two smooth charts 
	$\bMbar$ and $\Rn$. The \em sphere at infinity \em is the boundary of 
	$\Rnbar$ 
	$$
	\bM^\infty = 0\times \bS^{n-1}.
	$$
	%%Note that $\Rnbar$ is also semi-algebraically smoothly Diffeomorphic 
	%%to the unit closed ball $\bB_1^n$.
	If $S$ is a subset of $\Rn$, let $\ovS$ be its closure taken in $\Rnbar$.
	\begin{definition}\label{def:asymptotic-set}
		Let $S$ be a subset of $\Rn$. The \em asymptotic set of $S$ at infinity \em
		is the subset $S^\infty$ of $\bS^{n-1}$ defined as 
		$$
		0\times S^\infty := \ovS \cap \bM^\infty.
		$$ 
	\end{definition}

	The next notion is the avatar at infinity of a conic point. Their precise 
	relationship between the local notion and that at infinity will 
	be established in Lemma \ref{lem:conic-inversion}.
	\begin{definition}\label{def:conic-infty}
		A subset $X$ of $\Rn$ is \em conic at $\infty$, \em 
		if, either $X$ is bounded,  or the germ $(\ovX,0\times X^\infty)$ of 
		$(\Rnbar,0\times X^\infty)$ is not empty, and is a  smooth p-sub-manifold 
		germ of $(\Rnbar,0\times X^\infty)$.
		%%with boundary $\dd\ovX = 0\times X^\infty$. 
	\end{definition}
	\begin{remark}\label{rmk:germ-conic-infty}
		Following Definition \ref{def:conic-infty}, given a subset germ $(X,\infty)$ 
		the notion of $(X,\infty)$ be conic at $\infty$ is well defined.  
	\end{remark}

	\smallskip
	We recall that the inversion of $\Rn$ is the following rational isomorphism 
	$$
	\iota_n: \Rn\setminus\bbo \to \Rn\setminus\bbo, \;\; 
	\bx \mapsto \frac{\bx}{|\bx|^2}.
	$$
By definition $\bMbar = [\Rn,\bbo]$, and the mappings $\beta$ and $\bbl_\bbo$ admit an inverse only over $\Rn\setminus\bbo$. Thus
	$$
	\beta \circ \bbl_\bbo^{-1} = \bbl_\bbo\circ\beta^{-1} = \iota_n. 
	$$ 
	We observe that for a given subset $S$ of $\Rn$, the following identities hold
	true
	$$
	\iota_n(S\setminus \bbo)^\infty = S_\bbo S \;\; {\rm and} \;\;
	\bS_\bbo (\iota_n(S\setminus \bbo)) = S^\infty.
	$$
	If $\wh{Y^+}$ is a non-negative cone with vertex $\bbo$, we recall that
	$$
	\iota_n (\wh{Y}^+\setminus \bbo) = \wh{Y}^+\setminus \bbo.
	$$ 
	Let $S$ be a closed subset of $\Rn$, and define
	$$
	\wt{S} = \clos(\iota_n(S\setminus \bbo)).
	$$
	The next result is another variation of the many and very close similarities 
	between phenomenon at $\bbo$ and their analogues at $\infty$, 
	observed in our previous works \cite{CoGrMi1,CoGrMi2}.
	\begin{lemma}\label{lem:conic-inversion}
		The germ $(X,\infty)$ is conic at $\infty$ if and only if the 
		germ $(\wt{X},\bbo)$ is conic at $\bbo$.
	\end{lemma}
	\begin{proof}
		We will just show one implication, the converse one will proceed from 
		the symmetric argument without any additional difficulty.
		
		\medskip
		Assume $(X,\infty)$ is conic at $\infty$. Thus there exists a positive
		radius $R$ such that $X_\gR = X\setminus B_R^n$ is smooth sub-manifold with smooth compact boundary $X_R := X\cap \bS_R^{n-1}$. Thus 
		$\ovX$ is a smooth sub-manifold of $\bMbar$ with (compact) boundary 
		$\beta^{-1}(X_R) \cup 0\times X^\infty$, and the germ $(\ovX,0\times X^\infty)$ is a $p$-submanifold of $(\bMbar,\bM^\infty)$. We observe that $\beta^{-1}(\wh{X^\infty}^+\setminus \bbo)$ is the cylinder 
		$(0,\infty] \times X^\infty$.
		
		\medskip
		Since $\iota_n = \bbl_0 \circ\beta^{-1}$, we get 
		$$
		\clos(\bbl_\bbo^{-1}(\wh{\bS_\bbo\wt{X}}^+\setminus\bbo)) = \Rgo \times 
		X^\infty = \Rgo \times \bS_\bbo\wt{X},
		$$
		and since $(\iota_n(X_\gR),\bbo) = (\wt{X}\cap \bB_{1/R}^n \setminus\bbo,\bbo)$ is a smooth sub-manifold germ, $\bbo$ is a conic point of $\wt{X}$.
	\end{proof}
	%
	%
	
%%	The last ingredient needed to present the variation of the main 
%%	result in the Euclidean context is the following
	%
	%
	\begin{definition}\label{def:conic-Rn}
		A subset $X$ of $\Rn$ is a \em conic singular sub-manifold, \em if (i) it is closed,  (ii) the singular locus $X_\sing$ is finite or empty, (iii) each point of $X_\sing$ is conic, and  (iv) it is conic at $\infty$.
	\end{definition}

	\smallskip
	Let $\omg$ be the  north-pole $(0,\ldots,0,1)$ of $\bS^n \subset 
	\R^n\times\R$. The inverse of the 
	stereographic projection centred at $\omg$
	is the following semi-algebraic and smooth diffeomorphism 
	$$
	\sgm_n: \Rn\to\bS^n\setminus \omg, \;\;
	\bx \to \left(\frac{2\bx}{|\bx|^2+1},\frac{|\bx|^2-1}{|\bx|^2+1}\right).
	$$

	After Lemma \ref{lem:conic-inversion}, the following result is to be 
	expected.
	\begin{proposition}\label{prop:conic-spheric}
		A closed subset $X$ of $\Rn$ is a conic singular sub-manifold if and 
		only if $\clos(\sgm_n(X))$ is a conic singular sub-manifold of $\bS^n$. 
	\end{proposition}
	\begin{proof}
		If $X$ is compact the result is obvious. Thus we assume $X$ is unbounded.
		
		\smallskip
		Let $Z$ be $\clos(\sgm_n(X))$. Thus $Z = \sgm_n(X) \cup \omg$. Assume
		that $X$ is a conic singular sub-manifold of $\Rn$. Let $X_\sing$ be the
		set of singular points of $X$. Since $\sgm_n$ is smooth diffeomorphism, we 
		deduce that $Z\setminus \omg$ is a conic singular sub-manifold of 
		$\bS^n\setminus\omg$.
		
		\smallskip
		Consider the following mapping $\phi_n : \bB_\frac{1}{2}^n \to \bS^n$ defined
		as 
		$$
		\by \to \left(\frac{1}{1+r^2}\cdot 2\by,\frac{1-r^2}{1+r^2}\right)
		$$
		where $r =|\by|$. It is a smooth and semi-algebraic diffeomorphism
		onto its image, and 
		$$
		|\bx| \geq 2 \; \Longrightarrow \; \phi_n\circ\iota_n(\bx) = \sgm_n(\bx),
		$$
		thus $\phi_n^{-1}(Z)$ is conic at $\bbo$. Since 
		$$
		(\iota_n(X\setminus\bbo),\bbo) = (\sgm_n(\phi_n^{-1}(Z\setminus 
		\omg)),\bbo)
		$$
		Lemma \ref{lem:conic-inversion} implies that the subset  
		$\clos(\iota_n(X\setminus\bbo))$ is conic at $\bbo$, we deduce that $Z$ is 
		conic at $\omg$.
		
		\medskip
		Assume $Z$ is a conic singular sub-manifold, thus $X$ is conic at infinity by Lemma \ref{lem:conic-inversion}, since $\phi_n^{-1}(Z)$ is conic
		at $\bbo$. Since $\sgm_n$ is a smooth diffeomorphism, we conclude that 
		$X$ is a conic singular sub-manifold.
	\end{proof}
	The main result of this section is the affine version of Theorem 
	\ref{thm:main-compact}.
	\begin{theorem}\label{thm:main-Rn}
		A connected conic singular sub-manifold $X$ of $\Rn$ is LNE.
	\end{theorem}
	\begin{proof}
		The result is obvious when $X$ is compact. We assume that $X$ is 
		unbounded. The singular locus $X_\sing$ is finite thus is bounded. 
		Since $X$ is conic at infinity, there exists a positive radius $R_0$ such that $X_\sing$ is contained in the open ball $B_{R_0}^n$ and 
		$$
		X_\gR = X \setminus B_R^n 
		$$
		is a smooth sub-manifold with smooth compact boundary
		$X_R = X \cap \bS_R^{n-1}$ whenever $R\geq R_0$. Since $X$ is 
		connected, we can choose $R_0$ so that 
		$$
		X_\lR = X \cap \bB_R^n
		$$
		is connected for $R\geq R_0$.
		Since a germ of sub-manifold with boundary is locally LNE at the specified 
		point \cite[Corollary 2.7]{CoGrMi1}, and since a subset is locally LNE at each of its conic points, Lemma \ref{lem:compact-LNE}
		yields that $X_\lR$ is LNE, whenever $R\geq R_0$.
		
		By Proposition \ref{prop:conic-spheric}, the connected sub-manifold 
		$\clos(\sgm_n(X))$ is conic singular, thus LNE in $\bS^n$ by Theorem 
		\ref{thm:main-compact}, thus LNE in $\R^{n+1}$ since $\bS^n$ is LNE.
		Define  
		$$
		Z(R) := \sgm_n(X_\gR) \cup \omg = Z \cap \left\{\bz =(\bz',t) : t \geq 
		\frac{R^2-1}{R^2+1} \right\} \;\; {\rm and} 
		\;\; Z_R := \sgm_n(X_R).
		$$
		Since we can assume that $R_0\geq 2$, using the mapping 
		$\phi_n$
%%($=\sgm_n\circ\iota_n|_{\bB_{\frac{1}{2}}^n}$),
		introduced in the proof of Proposition \ref{prop:conic-spheric},
		we deduce by Claim \ref{claim:graph-origin} that 
		$$
		\wt{X}_{\leq \frac{1}{R}} = \clos(\iota_n(X_\gR)) = \phi_n^{-1}(Z(R)) 
		$$
		is (outer) bi-Lipschitz homeomorphic to the (compact truncated) non-negative 
		cone
		$$
		\wh{X^\infty}^+\cap \bB_\frac{1}{R}^n,
		$$
		which is LNE (in $\Rn$) by Corollary \ref{cor:cone-LNE}, since $X^\infty$
		is a compact sub-manifold of $\bS^{n-1}$. This bi-Lipschitz homeomorphism
		$$
		h_\bbo :\wt{X}_{\leq \frac{1}{R}} \; \to \; \wh{X^\infty}^+\cap 
        \bB_\frac{1}{R}^n
		$$
		satisfies $|h_\bbo(\by)| = |\by|$. Let 
		$$
		C(R) := \phi_n (\wh{X^\infty}^+\cap \bB_\frac{1}{R}^n)
		$$
		which is LNE in $\bS^n$. 
		Since $\bS^n$ is LNE as a subset of $\R^{n+1}$, we deduce that 
		$$
		h:= \phi_n \circ h_\bbo \circ \phi_n^{-1} : Z(R) \; \to \; C(R)
		$$
		is outer bi-Lipschitz. Since $C(R)$ is LNE, so is $Z(R)$.
		We check that 
		$$
		C(R) \setminus \omg = \sgm_n (\wh{X^\infty}^+\setminus B_R^n) 
		$$
		Since $C(R)$ and $Z(R)$ are outer bi-Lipschitz homeomorphic, 
		the main result of \cite{GrOl} implies that the homeomorphism
		$$
		h_\infty := \sgm_n^{-1} \circ h \circ \sgm_n : X_\gR \; \to \; 
		\wh{X^\infty}^+  \setminus B_R^n
		$$
		is outer bi-Lipschitz. 
		Lemma \ref{cor:cone-LNE} implies 
		that $\wh{X^\infty}^+  \setminus B_R^n$ is LNE, thus $X_\gR$ is LNE.
		
		\bigskip
		Assume that $X$ is not LNE and fix $R=R_1$. 
		Therefore there exists a sequence $(\bb_p)_p$ in $X_{\leq R_1}$ and 
		$(\bu_n)_n$ in $X_{\geq R_1}$ such that 
		$$
		d_p = \frac{d_\inn^X(\bu_p,\bb_p)}{|\bu_p - \bb_p|} \; \to \; \infty
		\;\; {\rm as} \;\; p \to \infty. 
		$$
		Up to passing to sub-sequences, we can assume that $\bb_p$ converge to 
		$\bb\in X_{\leq R_1}$ and $\bu_p$ either goes to $\infty$ or converge to
		$\bu \in X_{\geq R_1}$.
		
		Let $L_1$ be a LNE constant for $X_{\leq R_1}$ and $X_{\geq R_1}$. If $\by 
		\in X_{R_1}$, the following estimate holds true
		$$
		d_\inn^X(\bu_p,\bb_p) \leq L_1 |\bu_p - \by| + L_1 |\bb_p - \by|.
		$$
		This estimates implies that $\bb \in X_R$ and $(\bu_p)_p$ must converge 
		to $\bu = \bb$. 
		Thus we can assume that both sequences are contained in $X_{\leq R_1 +1}$.
		Since 
		$$
		d_\inn^{X_{\leq R_1 +1}} \; \leq \; d_\inn^X|_{X_{\leq R_1 +1}}
		$$
		we must have $d_p \to 0$. 
	\end{proof}
	\begin{remark}
		1) The main result of our paper \cite{CoGrMi2} states that a closed subset
		of $\Rn$, definable in an o-minimal structure, is LNE if and only if $\clos 
		(\sgm_n(X))$ is LNE. 
		%%(see \cite{Sam} too). 
		Thus the result for any such subset which is a connected 
		conic singular sub-manifold.
		\\
		2) Our initial proof of the fact that $X \setminus B_R^n$ is LNE was direct,
		but longer. 
		Indeed, since $\bMbar = [\Rn,\bbo]$ and $X$ is conic at $\infty$, Lemma 
		\ref{lem:collar-origin} holds true at infinity: there exists a 
		smooth diffeomorphism $\tau_\infty :\ovX \cap [0,r_0]\times \bS^{n-1}$ onto 
		$[0,r_0]\times X^\infty$, preserving the $r$-levels. Moreover, we obtain a 
		version at infinity of Claim \ref{claim:graph-origin}, proved exactly in the 
		same way, 
		namely that $\beta \circ \tau_\infty \circ \beta^{-1} : X \setminus B_R^n
		\to \wh{X^\infty}^+  \setminus B_R^n$ is bi-Lipschitz and link preserving, 
		thus LNE by Corollary \ref{cor:cone-LNE}.
	\end{remark}
	%
	%
	%
	%
	%
	%
	%
	%
	%
	%
	%
	%
	%
	%
	%
	%
	%
	%
	%
	%
	%
	%
	%
	%
	%
	%
	%
	%
	%
	%
	%
	%
	%
	%
	%
	% 
	%         *******************************************************
	%
	%
	%
	%
	%
	%
	%
	%
	%
	%
	%
	%
	%
	%
	%
	%
	%
	%
	%

	%
	%

\section{Application to projective algebraic sets}\label{section:PPoV}

Let $\K$ be either $\R$ or $\C$. 

Since we are interested in applications of Theorem \ref{thm:main-Rn},
we need to know how to go from $\field{KP}^n$, the natural space to 
compactify affine algebraic sub-varieties, to the compactification of the 
real Euclidean space with a sphere at infinity $\overline{\R^{n_\K}}$,
for $n_\K := \dim_\R \Kn$.

\bigskip
We identify $\R^{2n} = (\R^2)^n$ with $\Cn$  via the mapping 
$$
\bx = (x_1,y_1,x_2,y_2,\ldots,x_n,y_n) \to \bx^\C := 
(x_1+i y_1, \ldots,x_n+iy_n),
$$
whose inverse identifies $\Cn$ with $\R^{2n}$ as 
follows
$$
\bz = (z_1,\ldots,z_n) \mapsto \bz^\R := (x_1,y_1,x_2,y_2,\ldots,x_n,y_n)
$$
where $z_j = x_j + iy_j$ for each $j=1,\ldots,n$.

\bigskip
We embed the affine space $\K^n$ into the projective space 
$\bP^n := \mathbb{KP}^n$ as $\bx \mapsto [\bx:1]$. Let $\bH_\infty = \bP^n 
\setminus \K^n$ be the hyperplane at infinity of $\bP^n$, that is 
$\bH_\infty = \{[\by:0] :\by \in \K^n\setminus \bbo\}$.

\bigskip
We recall that $\overline{\R^{n_\K}} = \bMbar \sqcup \bbo$.  
We define the following mapping
$$
\pi_n^\K :\overline{\R^{n_\K}} \to \bP^n,\;\; \bx \to \pi_n^\K(\bx)
$$
as follows. If $\K = \R$, then
$$ 
\pi_n^\R(\bx) := 
\left\{
\begin{array}{lcl}
[\bbo:1] & {\rm if} & \bx = \bbo \\
{[\bu:r]} & {\rm if} & \bx = (r,\bu) \in \bMbar \\
\end{array}
\right.
$$
If $\bz\in\Cn^*$, then $\bz^\R = \beta(r,\bu)$ for some $(r,\bu) \in \bM$, 
that is with $r>0$, and thus $[\bz:1] = [\bu^\C:r]$. We define $\pi_n^\C$
as follows 
$$ 
\pi_n^\C(\bz) := 
\left\{
\begin{array}{lcl}
[\bbo:1] & {\rm if} & \bz = \bbo \\
{[\bu^\C:r]} & {\rm if} & (r,\bu) \in \bMbar \\
\end{array}
\right.
$$
Observe that the mapping $\pi_n^\K$ maps $\K^n$ into $\Kn$ as
$\bx \to [\bx:1]$. 
The following is clear
\begin{claim}\label{claim:projection-submersion}
The mapping $\pi_n^\K$ is a smooth semi-algebraic submersion.
\end{claim}
Given a subset $S$ of $\bP^n$, let $S^a := S \setminus \bH_\infty$ be its 
affine part.

\medskip
%%Two $\K$-vector subspaces $F,G$ of a $\K$-vector subspace $E$ of finite 
%%dimension are transverse if the sum $F+G$ has maximal dimension, that is
%%$\min (\dim E,\dim F + \dim G)$.
%%
The next result, which is well known, will allow to apply Theorem 
\ref{thm:main-Rn} and is just a consequence of transversality and Claim 
\ref{claim:projection-submersion}.
\begin{lemma}\label{lem:transverse}
Let $X$ be a smooth closed connected sub-manifold of $\bP^n$ of positive 
(real) dimension and codimension. Assume that $X\cap \bH^\infty$ is not 
empty and 
is a transverse intersection. Then the affine part $X^a$ 
is conic at $\infty$.
\end{lemma}
\begin{proof}
In the real case there is nothing to do, since at any boundary point
of $\bMbar$, the mapping $\pi_n^\R$ is a local diffeomorphism.

In the complex case, let $\bm = (0,\bu) \in \bM^\infty$. 
Let $\C_\bu$ be the complex line $\C\bu^\C \subset \Cn$. 
Let $\bS(\bu) = \C_\bu \cap 
\bS^{2n-1}$. Let $L_\bu$ be $T_\bu \bS(\bu)$, which is contained in the real
plane $\C_\bu$. Let $N_\bu$ be the real line of $\C_\bu$ orthogonal to 
$L_\bu$.  

The linear mapping $D_\bm \pi_n^\C$ maps isomorphically $T_\bm \bMbar = 
\R\dd_r \oplus T_\bu \bS^{2n-1}$ onto $T_{[\bu^\C:0]} \bP^n = \C_\bu \oplus 
T_{\bu^\C}\bH_\infty$ and maps isomorphically $T_\bu\bS^{2n-1}$ onto
$N_\bu\times T_{\bu^\C}\bH_\infty$. Since $X$ is transverse to 
$\bH^\infty$, the surjectivity of $D_\bm\pi_n^\C$ at any point of 
$\bm\in \bM^\infty$ yields the result.
\end{proof}

\bigskip
For $0 \leq d \leq q$, let 
$\bG (d,q)$ be the Grassmann manifold of $\K$-vector subspaces of
$\K^q$ of dimension $d$. 

Let $0 \leq e \leq q$. Let $\wt{\bG}(e,q)$ be the space of linear 
subspaces of $\bP^q$ of dimension $e$. 

Any linear subspace $F$ of $\K^{q+1}$ of dimension $e+1$, yields 
the unique linear subspace $\field{P}(F)$ of $\bP^q$ of dimension $e$, and 
conversely. This canonical correspondence provides the 
following identity
$$
\wt{\bG}(e,q) = \bG(e+1,q+1).
$$  
Given $\bx\in\K^{n+1}\setminus \bbo$, let $[\bx] \in \bP^n$ be the 
line direction of the $\K$-vector line $\K\bx$ of $\K^{n+1}$.

\bigskip\noindent
{\bf Notations:} \em From here on we take the following convention on 
cones: 
The $\K$-cone over the subset $S$ of $\bP^n$ is the subset of 
$\K^{n+1}$ defined as
$$
\wh{S} := \cup_{[\bx]\in S} \K\bx.
$$
\em

\bigskip\noindent
{\bf Convention:} \em A sub-variety of $\bP^n = \field{KP}^n$ is a 
$\K$-algebraic subset of $\bP^n$. 
\em

\bigskip
Let $X$ be a connected non-singular sub-variety of $\bP^n$ of $\K$-dimension
$1\leq d\leq n-1$. Let $\bx$ be a point of $X$ and let $T_\bx X$ be the 
tangent space of $X$, subspace of $T_\bx \bP^n$. The linear embedding of
the vector space $T_\bx \bP^n$ into $\bP^n$ induces a linear embedding of 
$T_\bx X$ into $\bP^n$. We define $\bP_\bx X$ as the closure of $T_\bx X$ 
taken in $\bP^n$. It is a $\K$-linear subspace of dimension $d$.
The projective tangent mapping of $X$ is defined as:
$$
\tau_X: X \to \wt{\bG}(d,n), \;\; \bx \to \bP_\bx X.
$$
Since $X$ is compact $\tau_X$ is proper. When $\K = \C$, the image 
$\tau_X(X)$ is therefore a sub-variety of $\wt{\bG}(d,n)$ of $\C$-dimension 
$d$ or lower. When $\K = \R$, the image $\tau_X(X)$ is contained in a 
sub-variety of $\wt{\bG}(d,n)$ of $\R$-dimension $d$ or lower.
In particular $\wt{\bG}(d,n) \setminus \tau_X(X)$ contains a non-empty 
Zariski open subset. We recall and give a proof of the following lemma
\begin{lemma}\label{lem:generic-cut}
There exists a Zariski open dense subset $\Omega_X$ of $\wt{\bG}(n-1,n)$
such that for each $H \in \Omega_X$, the intersection $X\cap H$ is everywhere 
transverse or empty (then $\K =\R$).
\end{lemma}
\begin{proof} 
Let $P$ be a linear subspace of $\bP^n$ of $\K$-dimension $p$. The 
following incidence variety
$$
V(P) := \{H\in \wt{\bG}(n-1,n) : P \subset H\}.
$$
is a linear subspace of $\wt{\bG}(n-1,n) = \bP^n$ of 
$\K$-dimension $n-p-1$ of $\wt{\bG}(n-1,n)$. Therefore the subset 
$$
W(X) :=  \cup_{P\in \tau_X(X)} V(P),
$$
which is algebraic in the complex case and semi-algebraic in the real case,
has $\K$-dimension $n-1$ or lower. 
Since $\wt{\bG}(n-1,n) \setminus W(X)$ contains a non-empty Zariski open
subset the result is proved.
\end{proof}
Given the hyperplane $H$ of $\bP^n$, we consider the affine space 
$$
\K_H^n := \bP^n \setminus H,
$$ 
naturally equipped with its Euclidean structure $eucl_H$. Observe that if 
$H_1$ and $H_2$ are hyperplanes of $\bP^n$, any unitary transformation of 
$\bP^n$ mapping $H_1$ onto $H_2$ produces a unitary $\K$-linear mapping 
between $(\K_{H_1}^n,eucl_{H_1})$ and $(\K_{H_2}^n,eucl_{H_2})$. 
If $S$ is any subset of $\bP^n$, the subset 
$$
S\setminus H \subset\K_H^n
$$
is the \em affine trace of $S$ w. r. t. the hyperplane $H$ of $\bP^n$. \em

\bigskip
Since the hyperplane $H$ is the hyperplane at infinity of the affine space 
$\K_H^n$, a straightforward consequence of Lemma \ref{lem:generic-cut}, Lemma
\ref{lem:transverse} and Theorem \ref{thm:main-Rn} is the following
result about generic affine traces of a non-singular projective sub-variety.
\begin{proposition}\label{prop:generic-cut-LNE-algebraic}
Let $X$ be a non-singular sub-variety of $\bP^n$ of positive $\K$-dimension and 
$\K$-codimension. 
There exists a non-empty Zariski open subset $\Omega_X$ of $\wt{\bG}(n-1,n)$ 
such that for each $H$ in $\Omega_X$, the affine trace $X\setminus H$ is 
non-singular and conic at $\infty$. Therefore each connected 
components of $X \setminus H$ is LNE in $\K_H^n$ w.r.t. $eucl_H$.
When $\K=\C$ and $X$ is irreducible, each such $X\setminus H$
is LNE.
\end{proposition}
%
%
%
%
%
%
%
%
%
%
%
%
%
%
%
%
%
%
%
%
%
%
%
%
%
%
%
% 
%         *******************************************************
%
%
%
%
%
%
%
%
%
%
%
%
%
%
%
%
%
%
%
%
\section{Application to affine algebraic sets}\label{section:APoV}
The approach of this section is complementary to the projective point of 
view of Section \ref{section:PPoV}. 
We will investigate the property of being LNE of affine algebraic 
sub-varieties
from their reduced equations. We will use the genericity results of the 
Appendix here.

\bigskip\noindent
{\bf Convention:} \em A sub-variety of $\Kn$ is a $\K$-algebraic subset
of $\Kn$. 
\em

\bigskip
Let $\K[\bx]$ be the $\K$-algebra of polynomial functions over $\Kn$. If $X$ 
is a sub-variety of $\K^n$, let $I(X)$ be the ideal of functions of $\K[\bx]$
vanishing along $X$. The subset $X_\sing$ is the sub-variety consisting of 
the singular points of $X$. The sub-variety $X$ is \em  non-singular \em if 
its singular locus $X_\sing$ is empty.
\begin{definition}\label{def:geom-CI}
Let $f_1,\ldots,f_p,$ be functions of $\K[\bx]$ and let 
$X$ be their common zero locus. The collection $f_1,\ldots,f_p,$ defines  
a \em set-theoretic complete intersection \em if
$$
X_\sing  = \{ \bx \in X \; : \; \wedge_{i=1}^p D_\bx f_i = \bbo\}.
$$
%%The functions $f_1,\ldots,f_p \in \K[\bx]$, with $2 \leq p \leq n-1$, define 
%%a \em set-theoretic complete intersection \em 
%%f $I(X)$ admits a minimal set of generators $f_1,\ldots,f_p$, with $p\geq 
%%1$, such that
%%$$
%%X_\sing  = \{ \bx \in X \; : \; \wedge_{i=1}^p D_\bx f_i = \bbo\} .
%%$$
\end{definition}
The condition in Definition \ref{def:geom-CI} is local analytic. 
It also implies that $p \leq n-1$ if $X_\sing$ is not empty. 
In particular
$$
\wedge_{i=1}^p D_\bx f_i \neq \bbo, \;\; \forall \bx \in X\setminus X_\sing,
$$
and thus the ($\K$-analytic) set germ $(X,\bx)$ is non-singular at $\bx$ and 
is of codimension $p$. Moreover the Jacobian criterion guarantees that 
$I(X)$ is generated by $f_1,\ldots,f_p$.

\bigskip
Let $\bx \to [\bx:1]$ be a fixed $\K$-linear embedding of $\Kn$ into $\bP^n$.
Let $\bH_\infty$ be the hyperplane at infinity.

Given an ideal $I$ of $\K[\bx]$, let $Z(I)$ be the zero-set of $I$. 
We will denote $P(I)$ the projective closure of $Z(I)$ in $\bP^n$,
and $P(I)^{zar}$ its Zariski closure in $\bP^n$, which contains $P(I)$. 
When $\K =\C$, both closures coincides.  

\bigskip
We recall that $\dd\overline{\R^{n_\K}}  = 0\times \bS^{n_\K-1}$.
Let $\Zbar(I)$ be the closure of $Z(I)$ in $\overline{\R^{n_\K}}$.
Following Lemma \ref{lem:transverse}, it is just the pull-back of $P(I)$
by $\pi_n^\K$. We also recall that 
$$
\Zbar(I) \cap \dd \overline{\R^{n_\K}} = 0 \times Z(I)^\infty.
$$

\subsection{Case of hypersurfaces}\label{subsec:hyper-LNE}
$ $ 

We will use the notations of Sub-section \ref{subsec:H}. 

\bigskip
We recall that $A(d,n)$ is the space of polynomials of $\K[\bx]$ of
degree $d$ or lower and that $H(d,n)$ is the subspace of $A(d,n)$ of 
homogeneous polynomials of degree $d$. 

Let $f \in A(d,n) \setminus A(d-1,n)$. The initial form $\ini(f)$ is the
homogeneous polynomial of degree $d$ appearing in the homogeneous 
decomposition of $f$. 

\bigskip
We obtain the following 
genericity result.
\begin{proposition}\label{prop:hyper-gen-LNE}
Assume $d\geq 1$.
There exists a Zariski open and dense subset $\cU(n,d)$ of $A(d,n)$, which 
does not intersect with $A(d-1,n)$, such that each $f \in  \cU(d,n)$
is irreducible, its zero locus $Z(f)$ is non-singular and conic at
$\infty$. Therefore each of its connected components is LNE.
In particular when $\K =\C$, the hypersurface $Z(f)$ is LNE.
\end{proposition}
\begin{proof}
If $Z(f)$ is compact, then $\K = \R$. If it is non-singular then each 
connected component is LNE by Lemma \ref{lem:compact-LNE}.

Take $\cU(d,n) = \Omg(d,n)$ of Proposition \ref{prop:aff-hyper-gen}.
In this case $f_d :=\ini(f) \in U(d,n)$ of Lemma \ref{lem:hyper-homog-gen}, 
it is irreducible and the projective hypersurface $P(f_d) \subset \bH_\infty$
is empty or non-singular. 
\\
Since $P(f) = Z(f) \cup P(f_d)$ is non-singular in $\bP^n$ and transverse 
to $\bH_\infty$, by Lemma \ref{lem:transverse}, the closure 
$\Zbar(f)$ is a p-sub-manifold of $\overline{\R^{n_\K}}$, and Theorem 
\ref{thm:main-Rn} applies.
\end{proof}
We present the following variation of the previous result.
\begin{corollary}\label{cor:generic-level-LNE}  
Let $f:\Kn \to \K$ be a polynomial function of positive degree $d$.
Assume that $f_d := \ini(f)$ lies in the Zariski open and dense 
subset $U(d,n)$ of Lemma \ref{lem:hyper-homog-gen}. Then any non-empty 
regular level is conic at infinity. Therefore each connected 
components of any non-empty regular level $f^{-1}(c)$ is LNE. 
If $\K =\C$, each level is irreducible.
\end{corollary}
\begin{proof}
Let $c$ be a regular value of $f$ and let $F_c := \overline{f^{-1}(c)}\subset
\overline{\R^{n_\K}}$ and $F_c^\infty := f^{-1}(c)^\infty$.

If $f^{-1}(c)$ is compact, the announced result follows 
from Lemma \ref{lem:compact-LNE}.

If $f^{-1}(c)$ is unbounded then $F_c$ is a smooth sub-manifold with boundary 
$0\times F_c^\infty$ in $\overline{\R^{n_\K}}$. By choice of $f_d$, it
is a p-sub-manifold of $\overline{\R^{n_\K}}$. We conclude again
with Theorem 
\ref{thm:main-Rn}.
\end{proof}

\subsection{Case of set-theoretic complete intersections}\label{subsec:GIC-LNE}
$ $ 

We will use the notations and results of Sub-section \ref{subsec:CI}.

\medskip
Let $\bd = (d_1,\ldots,d_p)$ where $n-1 \geq p \geq 1$ and $d_1 \geq \ldots
\geq d_p \geq 1$. 
We recall that 
$$
\bA(\bd,n) := A(d_1,n) \times\ldots\times A(d_p,n), \;\;
{\rm and}  \;\;
\bH(\bd,n) := H(d_1,n)\times\ldots\times H(d_p,n). 
$$
Given $\bbf = (f_1,\ldots,f_p) \in \bA(\bd,n)$, let
$$
\ini(\bbf) = (\ini(f_1),\ldots,\ini(f_p)) \;\; {\rm and} \;\;
Z(\bbf) := \cap_{i=1}^p Z(f_i).
$$
The main result of this sub-section is about set theoretic
complete intersections.
\begin{proposition}\label{prop:CI-gen-LNE} 
There exists a Zariski open and dense subset $\cV(\bd,n)$ of $\bA(\bd,n)$ 
such that for $\bbf\in\cV(\bd,n)$ the variety $Z(\bbf)$ is non-singular
and conic at $\infty$. Thus, each of its connected components are 
LNE. Furthermore $Z(\bbf)$ is irreducible if $p\geq n-2$. In particular 
$Z(\bbf)$ is LNE if $\K = \C$ and $p\leq n-2$. 
\end{proposition}
\begin{proof} 
Take $\cV(\bd,n) =  \Omg(\bd,n)$ of Proposition \ref{prop:aff-CI-gen}.

If $Z(\bbf)$ is compact then it is non-singular and thus the result is true
by Lemma \ref{lem:compact-LNE}.

If $Z(\bbf)$ is not compact, its closure $\Zbar(\bbf)$ in 
$\overline{\R^{n_\K}}$ is a smooth sub-manifold with boundary $0\times 
Z(\bbf)^\infty$ and transverse to $\dd \overline{\R^{n_\K}}$ by Proposition 
\ref{prop:aff-CI-gen} and Lemma \ref{lem:transverse}. Theorem 
\ref{thm:main-Rn} yields the results.
\end{proof}
A variation of this result is the following
\begin{corollary}\label{cor:CI-levels-LNE}  
Let $\bbf \in \bA(\bd,n)$ such that $\ini(\bbf)$ lies in $V(\bd,n)$ of 
Proposition \ref{prop:CI-homog-gen}. Then each connected component of 
each non-empty regular level $\bbf^{-1}(\bc)$ is LNE. If $\K =\C$, all
regular levels are LNE.
\end{corollary}
\begin{proof}
Same proof as that of Proposition \ref{prop:CI-gen-LNE}
\end{proof}
%
%

%
%
%
%
%
%
%
%
%
%
%
%
%
%
%
%
%
%
%
%
%
%
%
%
%
% 
%         *******************************************************
%
%
%
%
%
%
%
%
%
%
%
%
%
%
%
%
%
%
%
%
\section{Application to complex ICIS germs}\label{section:SAtICISG}
A by-product of the previous sections and the appendix, is the 
LNE result 
below for complex ICIS germs. Presenting also real statements would be 
at the cost of introducing further notations and cases to deal with the 
problem of non-negative homogeneous polynomials. Moreover the proofs would be 
the same as the complex one.

\bigskip
Let $\cO_n$ be the local $\C$-algebra of $\C$-analytic function germs 
$(\C^n,\bbo) \to \C$. Let $\mfrk_n$ be its maximal ideal. A function germ 
$f\in \cO_n$ writes as the formal power series
$$
f = \sum_{k\geq 0} f_k
$$
where $f_k \in H(k,n)$. Let $m_f$ be the multiplicity of $f$, that is the 
non-negative integer $d$ such that 
$$
f \in \mfrk_n^d \setminus \mfrk_n^{d+1}
$$
with the convention that $m_0 = \infty$. If $d = m_f < \infty$, we define
$$
\ini_\bbo(f) := f_d \in H(d,n)\setminus 0.
$$
If $I$ is any ideal of $\cO_n$,
let $(Z(I),\bbo)$ be the germ at $\bbo$ of its vanishing locus.
Let
$\Zbar_\bbo(I)$ be the strict transform of $Z(I)$ by $\bbl_\bbo$.

\medskip
Let $p \in \{1,\ldots,n-1\}$. Let $\bd = (d_1,\ldots,d_p) \in \N$ such that
$d_1\geq \ldots \geq d_p \geq 1$. Given $i \in \{1,\ldots,n\}$, we denote
$$
\bd + \bun_i := (d_1, \ldots,d_{i-1},1+d_i,d_{i+1},\ldots,d_p).
$$
Let us also denote 
$$
\mfrk(\bd,n) := \mfrk_n^{d_1}\times \ldots \times \mfrk_n^{d_p}.
$$
For $\bbf = (f_1,\ldots,f_p) \in \mfrk(\bd,n)\setminus \cup_{i=1}^n 
\mfrk(\bd+\bun_i,n)$, we define $\ini_\bbo(\bbf) \in \bH(\bd,n)$ as
$$
\ini_\bbo(\bbf) := (\ini_\bbo(f_1),\ldots,\ini_\bbo(f_p)).
$$
We recall that a germ $(X,\bbo) \subset (\Cn,\bbo)$ is an ICIS germ of 
codimension
$p$ (isolated complete intersection singularity) if there exists 
$\bbf = (f_1,\ldots,f_p)$ in some $\mfrk(\bd,n)$ such that 
$$
I(X) = (f_1,\ldots,f_p)
$$
and moreover for $\bx \in (X\setminus \bbo,\bbo)$, the following holds true
$$
\wedge_{i=1}^p D_\bx f_i \neq \bbo.
$$

\bigskip
%%If $\bd = d \in N$, that is $p=1$, let $\bH(\bd,n) := H(d,n)$. 
%%
Applying the results of the Appendix and Lemma 
\ref{lem:collar-origin} yields the 
following genericity result. The proof is only sketched 
since the arguments are the same that allow to
obtain Proposition \ref{prop:CI-gen-LNE}.
\begin{proposition}\label{prop:ICIS-germ-LNE}
Given $p  \in \{1, \ldots,n-1\}$ and $\bd$, there exists a Zariski open dense 
subset $\Ufrk(\bd,n)$ of $\bH(\bd,n)$ such that for each $\bbf \in 
\mfrk(\bd,n)$ such that $\ini_\bbo(\bbf)\in \Ufrk(\bd,n)$, the following properties are satisfied: 
(i) $(Z(\bbf),\bbo)$ is an ICIS; 
(ii) The origin is a conic singular point.
Therefore there exists a positive 
radius $r_\bbf$ such that for positive $r\leq r_\bbf$, the subset
$Z(\bbf)_{\leq r}$ is LNE.
\end{proposition}
\begin{proof}
Let $X := Z(\bbf)$ and let $\Ufrk(\bd,n) := V(\bd,n)$ of Proposition 
\ref{prop:CI-homog-gen}. 
First, by the results of the Appendix, observe that the projective 
variety $P(\ini_\bbo(\bbf))\subset \bP^{n-1}$ is irreducible and non-singular 
if $p\leq n-2$ or consists of finitely many points when $p=n-1$. Moreover 
$\crit(\bbf) \cap X$ is contained in $\{\bbo\}$, in other words
the germ 
$$
([X,\bbo],\ff ([\R^{2n},\bbo])) 
$$
is that of a p-sub-manifold of $([\R^{2n},\bbo],\ff ([\R^{2n},\bbo])$. Thus 
the origin is a conic singular point of $X$. Since the germ $(X,\bbo)$ admits 
sub-analytic representatives, the results follows from \cite[Corollary 4.6]{CoGrMi2}. 
\end{proof}
%
%
%
%
%
%
%
%
%
%
%
%
%
%
%
%
%
%
%
%
%
%
%
%
%
%
% 
%         *******************************************************
%
%
%
%
%
%
%
%
%
%
%
%
%
%
%
%
%
%
%
%
\appendix
\section{Some genericity results}\label{section:appendix}
We gather here, with proofs, variations of known genericity results 
about real and complex polynomials mappings. 
The case of real and complex plane curves which are LNE is elementary 
in the real case and follows from \cite{Tar,CoGrMi1} in the complex case.

\medskip\noindent
{\bf Hypothesis:} \em We assume in what follows that $\dim_\K \Kn = n \geq 
3$. \em

\subsection{Hypersurfaces}\label{subsec:H}
$ $

Let $\K[\bx] := \K[x_1,\ldots,x_n]$ be the $\K$-algebra of polynomial 
functions over $\K^n$. Let 
$$
A(d,n) := \{f \in \K[\bx] : \deg f \leq d\}.
$$
For $f \in A(d,n)$, let $Z(f)$ be its zero locus in $\K^n$ and $\crit(f)$
be its critical locus.

\bigskip
Let $H(d,n)$ be the vector subspace of $A(d,n)$ consisting of the homogeneous
polynomials of degree $d$. 
The zero locus $Z(g)$ of $g \in H(d,n)$ is a $\K$-cone of $\K^n$, perhaps 
reduced to the origin (in which case  $\K=\R$). When it is not empty, its 
quotient by the homogeneous action of $\K^*$ is the sub-variety 
$\bP Z(g)$ of $\bP^{n-1}$ ($=\bH_\infty$). In other words
$$
\bP Z(g) = \wh{P(g)} 
$$
the $\K$-cone over $\bP Z(g)$, with the convention that $\wh{\emptyset} = 
\bbo$, i.e. when $\bP Z(g)$ is empty. 

\medskip
Any polynomial $f \in A(d,n)\setminus A(d-1,n)$ decomposes uniquely 
in the direct sum
$$
A(d,n) = \bigoplus_{k=0}^d H(k,n),
$$
as 
$$
f = f_0 + f_1 + \ldots + f_d
$$ 
with $f_k \in H(k,n)$ for $k=0,\ldots,d$, and $f_d \neq 0$. The initial 
form of $f\in A(d,n)\setminus A(d-1,n)$ is
$$
\ini(f) := f_d.
$$
The affine embedding $e_n:\Kn \to \bP^n$, defined as $\bx \mapsto [\bx:1]$, 
induces the following $\K$-linear isomorphism $\wt{e}_n : A(d,n) \to 
H(d,n+1)$:
\begin{equation}\label{eq:lin-embed}
f \to \wt{e}_n(f), \;\; {\rm where} \;\; \wt{e}_n(f) (\bx,z ) :=
 \sum_{k=0}^d z^kf_{d-k}(\bx).
\end{equation} 
In other words $\wt{e}_n(f) \circ e_n = f$ for each $f\in {A(d,n)}$.

\bigskip
We recall that
$$
\dim_\K H(d,n) = N(d,n) = C_{n-1}^{d+n-1}.
$$
The following result is well known.
\begin{lemma}
Let $(n,d)$ be such that $n\geq 3$ and $d\geq 1$.
The subspace $U(d,n)^0$ of irreducible polynomials of $H(d,n)$ is Zariski 
open and not empty.
\end{lemma}
\begin{proof}
If $d =1$, then $U(d,n)^0 = H(d,n)\setminus \bbo$, where $\bbo$ means the 
null polynomial. 

\medskip
If $n-1 = d = 2$, we find $N(2,3) = 6 = 2 N(1,3)$. 
Let $R := H(1,3) \times H(1,3)$. Thus 
$$
R \setminus \bbo = \{(a_1x_1+a_2x_2+a_3x_3)\cdot (b_1x_1 + b_2x_2 + b_3 x_3), 
\; (\ba,\bb)  \in (\K^*)^3\times (\K^*)^3\}.
$$

Since the pair $(\lbd \ba,\lbd^{-1}\bb)$ gives the same polynomial of 
$R\setminus \bbo$ independently of $\lbd\in \K^*$, the space $R$ has positive
codimension. 

\medskip
Let $q:=n-1 \geq 2$ and consider the following polynomial
$$
Q(T) := \Pi_{k=1}^{q}(k+T).
$$
Therefore we see that
$$
N(d,n) = \frac{Q(d)}{Q(0)} = \Pi_{k=1}^q \frac{d+k}{k}.
$$
Let $d = e + f \geq 2e \geq 2$. Thus 
$$
D(f,e) := N(d,n) -  N(f,n) - N(e,n) = \frac{1}{q!}[Q(d) - Q(f) - Q(e)].
$$
Since 
$$
Q(d) = Q(e+f) \geq  Q(f) + e \cdot Q(e+f)\sum_{k=1}^q\frac{1}{k+d} .
$$
We deduce that 
$$
Q(d) - Q(f) - Q(e) \geq Q(e+f) \sum_{k=1}^q \frac{e}{k+d} - q \cdot 
\frac{Q(e)}{q}. 
$$
Observe that 
$$
Q(e+f) \cdot \frac{e}{d+k} \cdot \frac{q}{Q(e)} = \frac{e\cdot q}{e+k}\cdot 
\Pi_{l\neq k} \frac{l+d}{l+e}  \geq \frac{e\cdot q}{k+e}. 
$$
Since $q \geq 2$ and $e\geq 1$, $qe < e+q$ if and only if $e=1$.   
Whenever $e\geq 2$ we find that $D(f,e) >0$.

If $e = 1$, since $f\geq 2$ if $q =2$, then
$$
Q(1+f) - Q(f) - Q(1) = q (2+f) \cdots (q+f) - (q+1)!  > 0.
$$

Let $d_1\geq d_2\geq \ldots \geq d_p \geq 1$ be integer numbers such that
$$
d_1 + d_2 + \ldots + d_p =d
$$
for a positive integer $p\geq 2$. Let 
$$
R_p := H(d_1,n) \times H(d_2,n) \times \ldots \times H(d_p,n).
$$

We get the result, for the announced range of pairs $(n,d)$, by an induction 
on $p$ showing that $R_p$ has positive codimension in $H(d,n)$.
\end{proof}
\begin{remark}
If $f \in H(d,n) \setminus \bbo$ is irreducible, then $f+g$ is irreducible 
for any polynomial $g\in A(d-1,n)$. The converse is 
obviously false. 
\end{remark}
Since any polynomial $f$ of $H(d,n)$ satisfies the Euler identity
$$
D_\bx f \cdot \bx = d \cdot f(\bx), \;\; \forall \bx \in \Kn,
$$
the critical locus of $f$ is contained in its $0$-level
$Z(f)$.

\medskip
Given $\aph\in \N^n$, we write 
$$
\bx^\aph = \Pi_{i=1}^n x_i^{\aph_i}
$$
with the convention that $\bx^\bbo = 1$ for all $\bx \in \K^n$. 
Any polynomial $f \in A(d,n)$ writes as
$$
f (\bx) = \sum_{|\aph|\leq d} a_\aph \bx^\aph.
$$
We will at times identify the polynomial $f$ of $A(d,n)$ with its 
coefficients $\ba = (a_\aph)_{|\aph| \leq d}$. 
We will write $\ba(\bx)$ for the value of the polynomial at $\bx$.

\bigskip\noindent
Let $\bun_j \in \N^n$ be defined as 
$$
(\bun_j)_k = \dlt_{j,k}
$$
where $\dlt_{j,k}$ is the Kronecker symbol and $j,k=1,\ldots,n$.
Denoting $\dd_i$ the partial derivative in $x_i$, we get 
$$
\dd_i f = \dd_i\ba = (\aph_i a_\aph)_{|\aph| \leq d, \aph_i \geq 
1} = ((\beta_i+1)a_{\beta+\bun_i})_{|\beta|\leq d-1}\in A(d-1,n).
$$
\begin{lemma}\label{lem:hyper-homog-gen}
There exists a Zarisiki open and dense subset $U(d,n)$ of $H(d,n)$
such that each $f\in U(d,n)$ is irreducible and $Z(f) \cap \crit(f) 
\subset \{\bbo\}$. In particular the sub-variety $\bP Z(f)$ is either 
empty (and thus $\K = \R$) or it is a  non-singular irreducible hypersurface.
\end{lemma}
\begin{proof}
If $d = 1$, take $U(n,d) = H(n,d)\setminus \bbo$.

\medskip
Assume that $d\geq 2$.

Let $f = \ba$ and let $h_i(\ba,\bx) := \dd_i f (\bx)$.

Let us consider the following quasi-affine sub-variety of 
$H(d,n) \times (\Kn)^*$
$$
\Sgm := \{(\ba,\bx) \in U(d,n)^0 \times (\K^n)^* : h_1 = \ldots = h_n = 0\}.
%%\;\; {\rm and} \;\;
%%\Sgm^0 := \Sgm \setminus U(d,n)^0\times \bbo. 
$$

Let $(\ba,\bx) \in \Sgm$ such that $\bx \neq \bbo$. We can assume that $x_1
\neq 0$. Let $\aph^1,\ldots,\aph^n \in \N^n$ such that $|\aph^i| = d$ for 
$i=1,\ldots,n,$ and
$$
\aph_1^j = d\dlt_{1,j} \;\; {\rm while} \;\; 
\aph_i^j = (d-1)\dlt_{1,j} + \dlt_{i,j} \;\; {\rm for} \;\; j=1,\ldots,n,
\; i= 2,\ldots,n,
$$
where $\dlt_{k,l}$ is the Kronecker symbol. In particular we deduce
that 
$$
(\dd_{a_{\aph^1}} h_1) (\ba,\bx) = d \cdot x_1^{d-1} \;\; {\rm and} \;\; 
(\dd_{a_{\aph^i}} h_i) (\ba,\bx) = x_1^{d-1}
$$
and therefore 
$$
\wedge_{i=1}^n Dh_i (\ba,\bx) = d \cdot x_1^{d(d-1)} \cdot \wedge_{i=1}^n 
Da_{\aph^i} + \omg(\ba,\bx)
$$
where $\omg(\ba,\bx)$ is a $n$-form with null coefficient along 
$\wedge_{i=1}^n Da_{\aph^i}$. Therefore the $\K$-analytic subset 
germ $(\Sgm,(\ba,\bx))$ is non-singular and has codimension $n$.
Since $U(d,n)^0\times\bbo$ has $\K$-codimension $n$, we deduce 
that 
$$
{\rm codim}_\K \Sgm = n.
$$
Let $\pi_0 : \Sgm \to U(d,n)^0$ be the projection onto the second factor. 
Let 
$$
\pi_0^{-1}(\ba) =: \ba \times \Sgm_\ba. 
$$
Since $\Sgm_\ba \cup \bbo$ is a $\K$-cone of $\K^n$, and if $f =\ba \in 
U(d,n)^0$, then
$$
\crit(f) = \Sgm_\ba \cup \bbo \subset Z(f).
$$
Since $\pi_0$ is a regular mapping, the subset $K_0(\pi_0)$ of its critical 
values is $\K$-constructible in $U(d,n)^0$ of positive $\K$-codimension,
therefore its Zariski closure $K_0(\pi_0)^{zar}$ in $U(d,n)^0$ has also 
positive $\K$-codimension. Therefore $U(d,n) := U(d,n)^0 \setminus 
K_0(\pi_0)^{zar}$ is Zariski open and dense in $H(d,n)$.
\end{proof}
\begin{remark}\label{rmk:hyper-homog-gen}
The critical locus $\crit(f)$ of a polynomial $f \in A(d,n)$ for which 
$\ini(f) \in U(d,n)$ is necessarily compact. When $\K=\C$, such a 
polynomial $f$ admits at most finitely many critical points.
\end{remark}
A straightforward and useful consequence of Lemma \ref{lem:hyper-homog-gen} 
is the following 
\begin{corollary}\label{cor:hyper-homog-gen}
For each $f \in U(d,n)$, there exists a positive constant $C$ such that
$$
|grad (f)(\bx)| \geq C \, |\bx|^{d-1} \;\; \forall \bx \in \Kn
$$
where $grad(f) := \sum_{i=1}^n \dd_i f \dd_{x_i}$.
\end{corollary}
\begin{proposition}\label{prop:aff-hyper-gen}
There exists a Zariski open and dense subset $\Omg(d,n)$ 
of $A(d,n)$ such that if  
$f \in \Omg(d,n)$ then $f$ is irreducible, $\ini(f) \in U(d,n)$, 
the zero locus $Z(f)$ is either empty (and thus $\K =\R)$ or a non-singular
hypersurface of $\K^n$, whose projective closure coincides in $\bP^n$ with 
its Zariski projective closure $\bP Z(\wt{e}_n(f))$ and moreover it is
transverse to $\bH_\infty$.
\end{proposition}
\begin{proof}
Since $A(d,n)$ is $\K$-linearly isomorphic to $H(d,n+1)$, the subset 
$V_1:= \wt{e}_n^{-1}(U(d,n+1)) \setminus A(d-1,n)$ is Zariski open
and dense in $A(d,n)$. Observe 
that for any $f \in V_1$, we have
$$
Z(f) = \bP Z(\wt{e_n}(f)) \cap \K^n.
$$
In particular we deduce that $Z(f) \cap \crit(f)$ is empty by definition 
of $U(d,n+1)$. 
When $Z(f)$ is not compact we have
$$
\bP Z(\wt{e_n}(f))\cap \bH_\infty = \bP Z(\ini(f)).
$$

Let $V_2 := U(d,n) + A(d-1,n)$ which is Zarsiki open and dense in $A(d,n)$.  
By definition of $U(d,n)$, given any $f \in V_2$, if $Z(f)$ is not
compact, then $\bP Z(\ini(f))$ is a non-singular hypersurface of $\bH_\infty$.
 
Let $f \in V_1\cap V_2$. 
The projective closure of the $\K$-cone $Z(\ini(f))$ in $\bP^n$
has at most an isolated singularity at the origin $[\bbo:1]$, and is 
transverse to $\bH_\infty$.
Corollary \ref{cor:hyper-homog-gen} implies that whenever $|\bx| \gg 1$, we 
find 
$$
grad(f)(\bx) = grad (\ini(f))(\bx) + o(|grad (\ini(f))(\bx)|).
$$
and thus since $|grad (\ini(f))(\bx)| \geq C|\bx|^{d-1}$, we deduce that 
$\bP Z(\tilde{e}_n(f))$ is transverse to $\bH_\infty$.
Then we take $\Omg(d,n) := V_1\cap V_2$.
\end{proof}
\subsection{Set theoretic complete intersections}\label{subsec:CI}
$ $

Let $p$ be a positive integer such that $n-1 \geq p \geq 2$.

Let $\bd = (d_1,\ldots,d_p)$ be a $p$-tuple of integers such that
$d_1\geq d_2 \geq \ldots \geq d_p \geq 1$.
We define the following vector space
$$
\bH(\bd;n) := H(d_1,n) \times \ldots \times H(d_p,n).
$$
Given $\bbf = (f_1,\ldots,f_p) \in \bH(\bd,n)$, 
let $Z(\bbf)$ be the zero locus of the mapping $\bbf :\K^n\to\K^p$, and
let $\bP Z(\bbf)$ be the sub-variety of $\bP^{n-1}$, base of the $\K$-cone 
$Z(\bbf)$. The following torus is Zariski, open and dense in $\bH(\bd;n)$
$$
\bH(\bd;n)^* := \{\bbf \in \bH(\bd;n) \, : \, f_i \not\equiv 0, \; 
i=1,\ldots,p \}.
$$
%%In particular we find 
%%$$
%%{\rm codim}_\K \; \bH(\bd,n) \setminus \bH(\bd,n)^* = N(d_p,n) \geq n
%%$$
%%and it is equal $n$ if and only if $d_p =1$. 
%%\tcbl{Do we use this remark ?}

\medskip
Define the following polynomial function over $H(d,n)\times \K^n$:
$$
P_d(\ba,\bx) := \sum_{|\aph|\leq d} a_\aph \bx^\aph.
$$
Observe that for all $\bx\in \K^n$ we find
$$
\dd_{a_\aph} P_d(\ba,\bx) = \bx^\aph.
$$
Consider the following sub-variety of the torus $\bH(\bd,n)^*\times (\Kn)^*$
$$
IC(\bd;n) := \{(\bbf,\bx) = (\ba^1,\ldots,\ba^p,\bx) \in \bH(\bd,n)^* 
\times (\K^n)^* \; : \; 
P_{d_1}(\ba^1,\bx) = \ldots = P_{d_p}(\ba^p,\bx) = 0\}.
$$
\begin{lemma}
1)The sub-variety $IC(\bd,n)$ of $\bH(\bd,n)^* \times (\Cn)^*$ is 
non-singular, equi-dimensional, and of $\K$-codimension $p$.

\smallskip\noindent
2) There exists a Zariski open dense subset $V(\bd;n)^0$ of $\bH(\bd;n)^*$ such 
that for each $\bbf \in V(\bd;n)^0$ the zero locus $Z(\bbf)$ is either 
reduced to $\bbo$ (thus $\K=\R$) or admits at most an isolated singular 
point at $\bbo$, and is of $\K$-codimension $p$.
\end{lemma}
\begin{proof}
Let $(\bbf,\bx) \in IC(\bd,n)$, with $\bbf = (\ba^1,\ldots,\ba^p)$.
Since $\bx\neq \bbo$, for each $i=1,\ldots,p$, 
let $\aph^i \in \N^n$ such 
that 
$$
|\aph^i| = d_i \;\; {\rm and} \;\; \bx^{\aph^i}\neq 0 
$$
Observe that
$$
(\wedge_{i=1}^p D P_{d_i})(\ba,\bx) =  \bx^{\sum_{i=1}^p\aph^i} 
\wedge_{i=1}^p D_{\ba_{\aph^i}^i} + \omg(\bbf,\bx) \neq \bbo
$$
for $\omg$ a $p$-form whose coefficient along $\wedge_{i=1}^p Da_{\aph^i}^i$
is zero. Therefore $IC(\bd,n)$ is non-singular at $(\bbf,\bx)$, of 
$\K$-codimension $p$ at the point $(\bbf,\bx)$. 
Thus $Z(\bbf)$ is either reduced to $\bbo$ or is of codimension $p$, and
with at most a unique singular point, the origin $\bbo$.
The 
following mapping
$$
\pi(\bd;n): IC(\bd;n) \to \bH(\bd;n)^*
$$
is a regular mapping. Thus the locus $K_0(\pi(\bd;n))$ of its critical 
values is a $\K$-constructible subset of $\bH(\bd;n)^*$ of positive 
$\K$-codimension. Therefore its Zariski closure $K_0(\pi(\bd;n))^{zar}$
in $\bH(\bd;n)^*$ has also positive $\K$-codimension. We recall that  
$$
\pi(\bd,n)^{-1}(\bbf) = \bbf \times (Z(\bbf)\setminus \bbo). 
$$
Let 
$$
\Sigma := \pi(\bd,n)^{-1}(K_0(\pi(\bd,n)^{zar})).
$$
Since $\dim_\K IC(\bd,n) = \dim_\K \bH(\bd,n)^* + n-p$, the following 
subset   
$$
\pi(\bd,n) (IC(\bd,n)\setminus \Sigma) 
$$
is Zariski open and dense in $\bH(\bd,n)^*$.
\end{proof}
A first goal of this sub-section is to precise the previous result 
showing the following
\begin{proposition}\label{prop:CI-homog-gen}
There exists Zariski open and dense subset $V(\bd;n)$ of $\bH(\bd,n)^*$
such that if $\bbf \in V(\bd;n)$ then, the subset $Z(\bbf)\cap \crit(\bbf)$ 
is contained in $\{\bbo\}$. Moreover if $p\leq n-2$, the homogeneous ideal 
$(\bbf)$ generated by $f_1,\ldots,f_p$, is prime.
\end{proposition}

We introduce a few more notations, objects, and present 
several intermediary results to obtain Proposition \ref{prop:CI-homog-gen}.
Let 
$$
C(p,n) := \bigcup_{1\leq i_1<\ldots<i_{n-p} \leq n} \{\bx \in \Kn : x_{i_1} 
=  \ldots = x_{i_{n-p}} = 0\}
$$
the union of all $p$-planes of coordinates, thus of dimension $p$. 
It is a $\K$-cone over the projective sub-variety $B(p,n) \subset
\bP^{n-1}$, which is equi-dimensional of dim $p-1$. 
By definition of $V(\bd,n)^0$, the projective sub-variety $P(\bbf)$ of 
$\bP^{n-1}$ is non-singular and equi-dimensional of codimension $p$.
Thus we deduce that
$$
V(\bd,n)^1 := V(\bd,n)^0 \setminus \{\bbf : \dim_\K Z(\bbf) \cap C(p,n) \geq 
1\}
$$
is Zariski open and dense in $\bH(\bd,n)$. 
The following subset is Zariski open and dense in $H(d,n)$
$$
U(d,n)^* := \{f = (a_\aph)_{|\aph|=d} \in U(d,n) \; : \; a_\aph \neq 0, 
\; \forall \aph\}.
$$
The next subset is Zariski open and dense in $\bH(\bd;n)$
$$
W(\bd,n)^* := \left(U(d_1,n)^* \times \ldots U(d_p,n)^* \right)
\cap V(\bd;n)^1.
$$
We further define 
$$
IC(\bd,n)^* := IC(\bd,n) \cap W(d,n)^* \times (\Kn)^*
$$
as well as the following mapping
$$
\Psi (\bd,n) :IC(\bd,n)^* \to \K^p \times (\cL_n^*)^p,\;\;
(\bbf,\bx) \mapsto (P(\bbf,\bx),D_\bx\bbf).
$$
where $\cL_n = H(1,n)$, the space of $\K$-linear forms over $\Kn$. 
We obtain the following 
\begin{lemma}\label{lem:psi-bdn-submers} 
Assume $d_p \geq 2$, then $\Psi (\bd,n)$ is a submersion.
\end{lemma}
\begin{proof}
Consider the following mapping $\lbd(d,n) : H(d,n)^*\times (\Kn)^* 
\to \cL_n^*$ defined as
$$
(f,\bx) \mapsto D_\bx f :\bu \mapsto \sum_{i=1}^n \dd_i f(\bx) u_i.
$$

Let $(\bbf,\bx) \in IC(d,n)^*$. 
We can assume that 
$$
x_1\cdots x_p \neq 0.
$$
Since $\lbd(d_i,n)(\bbf,\bx) = (\dd_1f_i(\bx),\ldots,\dd_n f_i(\bx))$, we 
recall that 
$$
f_i (\bx) = \sum_{j=1}^p a_{d_i\bun_j}^i x_j^{d_i} + \sum_{k=p+1}^n
a_{(d_i-1)\bun_i + \bun_k}^i x_i^{d_i-1}x_k + \; {\rm  rest}.
$$
Let $E(p,n)$ be the set of pair $(i,j)$ with $i=1,\ldots,p$ and 
$j=1,\ldots,n$. Given $(i,j) \in E(p,n)$, let $\aph^{i,j} \in \N^n$ defined
as
$$
\aph^{i,j} = 
\left\{
\begin{array}{ccl}
d_i \bun_j & {\rm if} & j \leq p 
\\
(d_i-1)\bun_i + \bun_j  & {\rm if} & j \geq p + 1
\end{array}
\right.
$$
For $i,j = 1,\ldots,p$, and $k=1,\ldots,n$, we observe that
$$
\left(\dd_{a_{\aph^{i,j}}^i} \lbd(d_i,n)_k\right)(\bbf,\bx) = 
\dlt_{j,k}\cdot d_i \cdot x_j^{d_i-1},
$$
while for  $i= 1,\ldots,p$, $j = p+1,\ldots,n$, and $k=1,\ldots,n$,
we observe that
$$
\left(\dd_{a_{\aph^{i,j}}^i} \lbd(d_i,n)_k\right)(\bbf,\bx) = 
\dlt_{i,k}\cdot (d_i-1) \cdot x_i^{d_i-1}x_j + \dlt_{j,k} x_i^{d_i-1}.
$$
Thus we deduce the following identity holds true
$$
\left(\wedge_{(i,j)\in E(p,n)} \, D (\lbd(d_i,n)_j)\right)(\bbf,\bx) = 
(\Pi_{i=1}^p d_i^p) \cdot (\Pi_{i=1}^p x_i^{n(d_i-1)} )
\wedge_{(i,j)\in E(p,n)}  Da_{\aph^{i,j}}^i + \omg_\aph (\bbf,\bx)
\neq \bbo
$$
where $\omg_\aph$ is a $p\times n$-form with null component along 
$\wedge_{(i,j)\in E(p,n)} Da_{\aph^{i,j}}^i$.
\\
For $i=1,\ldots,p$, let $\beta^i\in \N^n$ be such that
$$
\beta^i = (d_i -1)\bun_1 + \bun_2. 
$$
Since $p\geq 2$, we get
$$
(\wedge_{i=1}^p D P_{d_i})(\bbf,\bx) = x_1^{\sum_{i=1}^p (d_i-1)}
x_2^p \wedge_{i=1}^p Da_{\beta^i}^i + \omg_\beta(\bbf,\bx) \neq \bbo
$$
where $\omg_\beta$ is a $p$-form with null component along $\wedge_{i=1}^p 
Da_{\beta^i}^i$.
\\
Since $d_i\geq 2$, we check that $\beta^i \neq \aph^{i,j}$ for each 
$j=1,\ldots,n$. The lemma is proved.
\end{proof}

\medskip\noindent
We can now go into the proof of Proposition \ref{prop:CI-homog-gen}.
\begin{proof}[Proof of Proposition \ref{prop:CI-homog-gen}]
We will distinguish three cases.

\medskip\noindent
$\bullet$ {\bf Case 1:} \em Assume $d_1 = \ldots = d_p =1$. \em In this case,
since the mapping $\bbf :\Kn \to \Kp$ is $\K$-linear, the open subset 
$V(\bd,n)$ is only the Zariski open subset of such mappings of rank $p$. 

\medskip
We recall that the determinantal sub-variety of $\cL_n^p$
$$
D(p,n) : = \{(\lbd_1,\ldots,\lbd_p) \in \cL_n^p \, : \, \wedge_{i=1}^p \lbd_i 
= \bbo  \}
$$
is a $\K$-cone of $\K$-dimension $(p-1)(n+1)$ (see \cite{BrVe}), that is of 
$\K$-codimension $n+1 - p$. Therefore $$
\cD(p,n) :=\bbo\times D(p,n) \subset \K^p\times \cL_n^p
$$
has codimension $n+1$.  

\medskip\noindent
$\bullet$ {\bf Case 2:} \em Assume that $d_p \geq 2$. \em 
Denoting $\cD(p,n)^* = \cD(p,n) \cap (\cL_n^*)^p$, we get
$$
\Sgm:= \Psi^{-1}(\bbo\times (\cL_n^*)^p) \cap \cD(p,n)^* = 
\{(\bbf,\bx) \in IC(\bd,n)^* : {\rm rk}\,D_\bx \bbf \leq p-1\}.
$$
Since the sub-variety $\cD(p,n)^*$ is a finite union of non-singular constructible sub-manifolds $S_1,\ldots,S_r$, Lemma \ref{lem:psi-bdn-submers} implies that 
$\Psi$ is transverse to each $S_k$.  
Thus $\Sigma$ has codimension $n+1$. Therefore the projection of 
$IC(\bd,n)^*\setminus
\Sigma$ onto $\bH(d,n)^*$ is Zariski open, since the projection of 
$IC(\bd,n)^*$ is $V(\bd,n)^1$.  

\medskip\noindent
$\bullet$ {\bf Case 3:} \em Assume $d_q >d_{q+1} = \ldots = d_p =1$ for 
some index $q \in \{1,\ldots,p-1\}$. \em Thus given $\bbf \in \bH(\bd,n)^*$,
the mapping $\bbf':=(f_{q+1},\ldots,f_p):\Kn\to \K^{p-q}$ is just a 
$\K$-linear mapping. The following subset 
$$
V(\bd,n)^2 := \{\bbf \in  V(\bd,n)^1 : {\rm rk}\,\bbf' = p-q\}
$$
is Zariski open. Let
$$
IC(\bd,n)^2 := IC(\bd,n)^*\cap V(\bd,n)^2 \times (\Kn)^*.
$$
Let us denote again $\lbd(d_i,n)(\bbf,\bx)$ for $D_\bx f_i \in \cL_n$.
Thus if $\bbf \in V(\bd,n)^2$, we find 
$$
(\wedge_{i=q+1}^p\wedge_{j=1}^n D \lbd(d_i,n)_j)(\bbf,\bx) \neq \bbo.
$$
For $i=1,\ldots,q$, let $\aph^{i,j} \in \N^n$ as in the proof of Lemma 
\ref{lem:psi-bdn-submers}. For $i = q+1,\ldots,p$, define 
$$
\aph^{i,j} = \bun_j
$$
And we check that
$$
\left(\wedge_{(i,j)\in E(p,n)} \, D (\lbd(d_i,n)_j)\right)(\bbf,\bx) = 
(\Pi_{i=1}^p d_i^p) \cdot (\Pi_{i=1}^p x_i^{n(d_i-1)} )
\wedge_{(i,j)\in E(p,n)}  Da_{\aph^{i,j}}^i + \omg_\aph (\bbf,\bx)
\neq \bbo
$$
where $\omg_\aph$ is a $p\times n$-form with null component along 
$\wedge_{(i,j)\in E(p,n)} Da_{\aph^{i,j}}^i$.
For $i=1,\ldots,q$, let $\beta^i$ as in the proof of Lemma 
\ref{lem:psi-bdn-submers}, so that
$$
(\wedge_{i=1}^q D P_{d_i})(\bbf,\bx) = x_1^{\sum_{i=1}^q (d_i-1)}
x_2^q \wedge_{i=1}^q Da_{\beta^i}^i + \omg_\beta(\bbf,\bx) \neq \bbo
$$
where $\omg_\beta$ is a $q$-form with null component along $\wedge_{i=1}^q 
Da_{\beta^i}^i$.
\\
Since $\bbf'$ has rank $p-q$, we know that $\wedge_{i=q+1}^p D_\bx f_i \neq 
\bbo$, thus we deduce that $\Psi$ is submersive at each point of 
$IC(\bd,n)^2$. We conclude as in Case 2 working with $IC(\bd,n)^2$.

\medskip
Let $(\bbf)$ be the ideal generated by $(f_1,\ldots,f_p)$ in $\K[\bx]$,
with $p\leq n-2$.

Assume first that $\K = \C$. Since $Z(\bbf)$ is non-singular outside the 
origin, of $\K$-codimension $p$ and $Z(\bbf)\cap \crit(\bbf)$ is contained
in $\{\bbo\}$, the ideal $(\bbf)$ is reduced. Since $Z(\bbf) \setminus 
\bbo$ is connected, and each of the $f_i$ is irreducible, the ideal
$(\bbf)$ is necessarily prime.

\smallskip
Assume that $\K=\R$.
Let $H(d,n,\K)$ be the $\K$-homogeneous polynomials of degree $d$
over $\Kn$. Since $\R$ is the real part of $\C$,
the space $H(d,n,\R)$ is the real part of $H(d,n,\C)$:
$$
H(d,n,\R) := \{(a_\aph)_{|\aph| =d} \in H(d,n,\C) : a_\aph \in \R, \; 
\forall \aph\}.
$$
If $\cU$ is a non-empty Zariski open subset of $H(d,n,\C)$, the subset 
$$
\cU \cap H(d,n,\R)
$$
is Zariski open in $H(d,n,\R)$ and not empty. Therefore we can assume
that $V(\bd;n,\R)$ is contained $V(\bd;n,\C) \cap \bH(\bd,n,\R)$,
and thus for $\bbf \in V(\bd;n,\R)$ the ideal $(\bbf)$ is prime
in $\C[\bx]$, thus prime in $\R[\bx]$. 
\end{proof}
The following 
consequence of the proof of Proposition \ref{prop:CI-homog-gen} is 
straightforward
\begin{corollary}\label{cor:CI-homog-gen}
Let $\bbf \in V(\bd;n)$ and such that $\bP Z(\bbf)\subset \bH_\infty$ is not 
empty. At every $[\bx:0] \in \bP Z(\bbf)$ the following holds true 
$$
(\wedge_{i=1}^p D_\bx f_i) \wedge Dz \neq \bbo.
$$
In particular $Z(\bbf)\cup \bP Z(\bbf)$, the Zariski closure of $Z(\bbf)$ in 
$\bP^n$, is transverse to $\bH_\infty$.
\end{corollary}

\bigskip
We define the following vector space
$$
\bA(\bd,n) := A(d_1,n)\times\ldots,\times A(d_p,n).
$$
As for the case $p=1$, using the homogenization, we obtain again a
$\K$-linear isomorphism
$$
\wt{\be}_n : \bA(\bd,n) \to \bH(\bd,n+1),\;\;
\bbf = (f_1,\ldots,f_p) \; \mapsto \; \wt{\be}_n(\bbf) := ( \wt{e}_n(f_1),\ldots, 
\wt{e}_n(f_p)).
$$
For $\bbf = (f_1,\ldots,f_p) \in \bA(\bd,n)$, we also define
$$
\ini(\bbf) := (\ini(f_1),\ldots,\ini(f_p)).
$$
The next result is the analogue of Proposition \ref{prop:aff-hyper-gen}
and admits a similar proof.
\begin{proposition}\label{prop:aff-CI-gen}
1) There exists a Zariski dense and open subset $\Omg(\bd,n)$ of $\bA(\bd,n)$ 
such 
that if  $\bbf \in \Omg(\bd,n)$ then $\ini(\bbf) \in V(\bd,n)$, the zero 
locus $Z(\bbf)$ is either empty (and 
thus $\K =\R)$ or a non-singular sub-variety of $\Kn$, whose projective 
closure coincides in $\bP^n$ with its Zariski projective closure 
$\bP Z(\wt{\be}_n(\bbf))$ and moreover it is transverse to $\bH_\infty$.

\smallskip\noindent
2) If $p\leq n-2$, the ideal $(\bbf)$ is prime.
%%\\
%%\tcbl{What about $(\bbf)$ prime when $\K =\R$, in particular when
%%$Z(\bbf)$ is compact ?}
\end{proposition}
%
%
%%\tcbl{What about $(\bbf)$ prime when $p=n-1$ and  $\K =\C$. ?}
%
%
\begin{proof}
Since $\bA(\bd,n)$ is $\K$-linearly isomorphic to $\bH(\bd,n+1)$, the 
following subset
$$
\cV_1 := \wt{\be}_{n+1}^{-1}(V(\bd,n+1)) \cap \{\bbf \in \bA(\bd,n) : 
\ini(\bbf)\in \bH(\bd,n)^*\}
$$
is Zariski open and dense in $\bA(\bd,n)$. For any $\bbf \in \cV_1$, we 
find
$$
Z(\bbf) = \bP Z(\wt{\be}_{n+1}(\bbf)) \cap \K^n.
$$
In particular we deduce that $Z(\bbf) \cap \crit(\bbf)$ is empty by 
definition of $V(\bd,n+1)$. 
When $Z(\bbf)$ is not compact we have
$$
\bP Z(\wt{\be}_{n+1}(\bbf)) \cap \bH_\infty = \bP Z(\ini(\bbf)).
$$
The following subset 
$$
\cV_2 := V(\bd,n) + \{\bbf \in A(\bd,n) : \exists \; i \in \{1,\ldots,n\} \;
: \; \deg f_i < d_i\}
$$
is Zariski open and dense in $\bA(\bd,n)$. By definition of $V(\bd,n)$, 
given any mapping $\bbf \in \cV_2$, if $Z(\bbf)$ is not
compact, then $\bP Z(\ini(\bbf))$ is a non-singular set theoretic complete 
intersection  of $\bH_\infty$.
 
Let $\bbf \in \cV_1\cap \cV_2$.  
The projective closure of the $\K$-cone $Z(\ini(\bbf))$ in $\bP^n$
has at most an isolated singularity at the origin $[\bbo:1]$, and is 
transverse to $\bH_\infty$.

Corollary \ref{cor:CI-homog-gen} implies that whenever $|\bx| \gg 1$, 
and $\bx \in Z(\bbf)$ we find
$$
|\wedge_{i=1}^p D_\bx \ini(f_i)|\;\geq \; C|\bx|^{D-p}
$$
where $D = \sum_{i=1}^p d_i$ and $C$ is a positive constant.
Since for $|\bx| \gg 1$, the following estimate holds true
$$
\wedge_{i=1}^p D_\bx f_i = \wedge_{i=1}^p D_\bx \ini(f_i) + 
o(|\bx|^{D-p}),
$$
we thus deduce that $\bP Z(\wt{\be}_n(\bbf))$ is transverse to $\bH_\infty$.
Then we take $\Omg(\bd,n) := \cV_1\cap \cV_2$.

\medskip\noindent
If $p\leq n-2$, the ideal $(\ini(\bbf))$ is prime, thus so is the 
ideal $(\bbf)$.
\end{proof}
%
%
%
%
%
%
%
%
%
%
%
%
%
%
%
%
%
%
%
%
%
%
%
%
%
%
%
%
%
%
%
%
%
%
%
%
%
% 
%         *******************************************************
%
%
%
%
%
%
%
%
%
%
%
%
%
%
%
%
%
%
%
%

%
%
%
%
%
%
%
%
%
%
%
%
%
%
%
%
%
%
%
%
%
%
%         *************************************************************
%
%
%
%
%
%
%
%
%
%
%
%
%
%
%
%
%
%

\end{document}